\documentclass[12pt]{amsart}
\usepackage{amscd}
\usepackage[mathscr]{eucal}
\usepackage{amsfonts}
\usepackage{amsmath}
\usepackage{amsthm}
\usepackage{amssymb}
\usepackage{latexsym}
\usepackage{tabularx,array}
\newfont{\msam}{msam10}

\newtheorem{theorem}[]{Theorem}
\newtheorem{proposition}[]{Proposition}
\newtheorem{corollary}[]{Corollary}
\newtheorem{lemma}[]{Lemma}
\theoremstyle{definition}
\newtheorem{definition}[]{Definition}
\newtheorem{remark}[]{Remark}

\newtheorem{prop}[]{Proposition}

\let\nc\newcommand

\nc{\la}{\label}

\def\bthm{\begin{theorem}}
\def\ethm{\end{theorem}}
\def\blemma{\begin{lemma}}
\def\elemma{\end{lemma}}
\def\bproof{\begin{proof}}
\def\eproof{\end{proof}}
\def\bprop{\begin{proposition}}
\def\eprop{\end{proposition}}
\def\bcor{\begin{corollary}}
\def\ecor{\end{corollary}}

\def\Z{\mathbb{Z}}
\def\R{\mathbb{R}}

\def\V{\mathcal{V}}

\def\A{\mathbb{A}}

\def\K{\mathbb{K}}

\def\c{\mathbb{C}}

\nc{\Hom}{{\rm{Hom}}}
\nc{\Ext}{{\rm{Ext}}}
\nc{\HOM}{\underline{\rm{Hom}}}
\nc{\EXT}{\underline{\rm{Ext}}}
\nc{\TOR}{\underline{\rm{Tor}}}
\nc{\End}{{\rm{End}}}
\nc{\GL}{{\rm{GL}}}
\nc{\PGL}{{\rm{PGL}}}
\nc{\SL}{{\rm{SL}}}
\nc{\PSL}{{\rm{PSL}}}
\nc{\Rep}{{\rm{Rep}}}
\nc{\ad}{{\rm{ad}}}
\nc{\dlim}{\varinjlim}

\newcommand{\rar}{\rightarrow}
\newcommand{\dk}{\mathbb D_2}
\newcommand{\hb}{\hbar}
\newcommand{\hh}{\text{HH}}
\newcommand{\W}{\mathbb W}
\newcommand{\Ww}{\mathcal W}
\newcommand{\Aa}{\mathcal A}
\newcommand{\Aad}{{\mathcal A}_{\text{Dunkl}}}
\newcommand{\C}{\mathcal C}
\newcommand{\cH}{\mathcal H}

\begin{document}

\title{Hochschild (co)homology of the Dunkl operator quantization of $\Z_2$-singularity.}
\author{Ajay Ramadoss}
\address{Departement Mathematik,
 ETH Z\"{u}rich, R\"{a}mistrasse 101, 8092 Z\"{u}rich}
\email{ajay.ramadoss@math.ethz.ch}
\author{Xiang Tang}
\address{Department of Mathematics, Washington University, St.
Louis, MO 63130, USA} \email{xtang@math.wustl.edu}
\maketitle

\begin{abstract}We study Hochschild (co)homology groups of the Dunkl
operator quantization of $\Z_2$-singularity constructed by Halbout and
Tang. Further, we study traces on this algebra and prove a local
algebraic index formula.
\end{abstract}

\section{Introduction}

Quantizations of symplectic orbifolds were constructed by many
authors, i.e. \cite{MM}, \cite{P}, and \cite{T}. The main idea in
the constructions is that the classical quantization methods on a
symplectic manifold $M$ can be made equivariantly with respect to the action of a
finite group $G$, and therefore can be generalized to the
orbifold $M/G$. Let us denote the algebra of quantized functions by
$\A_{M/G}((\hbar))$. With the efforts of many authors, i.e.
\cite{DE}, \cite{NPPT}, \cite{PPT}, \cite{PPT2}, \cite{PPTT}, we
have come to understand the Hochschild and cyclic (co)homology groups of and
algebraic index theory on $\A_{M/G}((\hbar))$.\\

In \cite{HT}, modeled by the Dunkl operator, Halbout and Tang
constructed an algebra $\mathfrak{A}_{M/\Z_2}((\hbar_1))((\hbar_2))$ as an interesting deformation of
the algebra $\A_{M/\Z_2}((\hbar))$ when $G=\Z_2$. This
deformation is a global version of a symplectic reflection algebra
introduced by Etingof and Ginzburg \cite{EG}. We devote this paper
to study Hochschild (co)homology groups of the algebra $\mathfrak{A}_{M/\Z_2}((\hbar_1))((\hbar_2))$. Such study could not only help in
understanding the algebra $\mathfrak{A}_{M/\Z_2}((\hbar_1))((\hbar_2))$, but also shed light on the relationship between
this algebra and the geometry of the singularity. Our hope is that our
efforts in this paper could eventually lead to a full generalization
of the algebra $\mathfrak{A}_{M/\Z_2}((\hbar_1))((\hbar_2))$ to
general $A_n$ type singularities.\\

Halbout and Tang's original construction (See Section 4
of~\cite{HT}) of $\mathfrak{A}_{M/\Z_2}((\hbar_1))((\hbar_2))$ is
through a global gluing procedure. Such a global construction makes
it difficult to compute the cohomology groups of this algebra. Our
solution to this problem is construct a presheaf of algebras on
$M/\Z_2$, the space of global sections of which defines the algebra
$\mathfrak{A}_{M/\Z_2}((\hbar_1))((\hbar_2))$. This crucial
improvement gives the key to compute the Hochschild homology of the
algebra $\mathfrak {A}_{M/\Z_2}((\hbar_1))((\hbar_2))$. It turns out
to coincide with shifted Chen-Ruan cohomology groups of $M/\Z_2$
with coefficients in $\c((\hb_1))((\hb_2))$. To compute the
Hochschild cohomology
$\text{HH}^{\bullet}(\mathfrak{A}_{M/\Z_2}((\hbar_1))((\hbar_2)),\mathfrak{A}_{M/\Z_2}((\hbar_1))((\hbar_2)))$
of the algebra $\mathfrak{A}_{M/\Z_2}((\hbar_1))((\hbar_2))$, we
were first tempted to apply Van Den Bergh duality \cite{Vdb1},
\cite{Vdb2} in the framework of bornological algebras following
\cite{DE}. However, it is not clear that even in this context,
$\mathfrak{A}_{M/\Z_2}((\hbar_1))((\hbar_2))$ has a resolution by
finitely generated bimodules. We therefore, proceed to prove that
there exists $\theta \in
\text{HH}_{2n}(\mathfrak{A}_{M/\Z_2}((\hbar_1))((\hbar_2)))$ such
that the cap product $$\theta \cap -:
\text{HH}^{\bullet}(\mathfrak{A}_{M/\Z_2}((\hbar_1))((\hbar_2)),
\mathfrak{A}_{M/\Z_2}((\hbar_1))((\hbar_2))) \rar
\text{HH}_{2n-\bullet}(\mathfrak{A}_{M/\Z_2}((\hbar_1))((\hbar_2)))$$
is an isomorphism of $\c((\hb_1))((\hb_2))$-modules. This is proven
by verifying a similar assertion for the algebra of quantum
functions on $M/\Z_2$.

With a good understanding of the Hochschild (co)homology groups of
the algebra $\mathfrak{A}_{M/\Z_2}((\hbar_1))((\hbar_2))$, we study
traces on this algebra, together with a related local algebraic
index theorem. Our main new discovery is an explicit trace formula
on the Dunkl-Weyl algebra $\dk((\hbar_1))((\hbar_2))$ generalizing
Fedosov's $\gamma$-trace \cite{PPT}. Using this local trace, we
define an global trace on the algebra
$\mathfrak{A}_{M/\Z_2}((\hbar_1))((\hbar_2))$ following the idea of
\cite{PPT}. We prove a local index formula for the evaluation of
this trace on $1$. This local index formula may be viewed as a
homological detector of the fact that
$\mathfrak{A}_{M/\Z_2}((\hb_1))((\hb_2))$ is a second quantization
of the ring of quantum functions on $M/\Z_2$. By the Hochschild
homology of $\mathfrak{A}_{M/\Z_2}((\hbar_1))((\hbar_2))$, we know
that there should be one more trace on this algebra that forms a
$\c((\hb_1))((\hb_2))$-basis for the space of traces on
$\mathfrak{A}_{M/\Z_2}((\hbar_1))((\hbar_2))$ along with the trace
that we have explicitly constructed. Unfortunately, we have not been
able to give a local explicit construction of the latter trace.
Instead, we succeed in giving an abstract construction of this trace
using functors in the derived category of sheaves. It would be very
interesting to have a better understanding of this trace, and we
plan to come back to this question in the future.

\subsection{Organization of this paper.} This paper is organized as follows. In Section 2, we study
properties of the Dunkl-Weyl algebra $\dk((\hbar_1))((\hbar_2))$. We
compute its Hochschild (co)homology groups. We prove an explicit trace formula on
$\dk((\hbar_1))((\hbar_2))$. In Section 3, we compute the Hochschild
(co)homology groups of the algebra $\mathfrak{A}_{M/\Z_2}((\hbar_1))((\hbar_2))$ with a refined construction of
this algebra. By the end of Section 3, we are left with some technical propositions that have to be proven in order to complete the proofs of the main results in this section. In Section 4, we study traces on $\mathfrak{A}_{M/\Z_2}((\hbar_1))((\hbar_2))$ and prove a local algebraic
index formula. We use our knowledge of trace densities to prove the technical propositions required to complete the main results of Section 3.\\

\noindent{\bf Acknowledgments:} We are especially grateful to Yuri Berest and Giovanni Felder for their encouragement, the very useful discussions we had with them as well as their important suggestions. We would also like to thank Damien
Calaque, Vasiliy Dolgushev and Gilles Halbout for
interesting discussions and suggestions. A.R. is supported by the Swiss National Science Foundation for the project ``Topological quantum mechanics and index theorems" (Ambizione Beitrag Nr.$\text{ PZ}00\text{P}2\_127427/1$). The contribution of A.R. towards this paper is part of work done towards this project. A.R. would also like to thank ETH Z\"{u}rich as well as Cornell University for providing excellent working conditions at various stages in this work. X.T. is partially supported by NSF grant
0900985.

\section{Formal computations.}

In this section, we compute the Hochschild homology of formal
analogs of the symplectic reflection algebras defined by ~\cite{EG}.
As a special case, we obtain the Hochschild homology of the
Dunkl-Weyl algebra $\dk((\hb_1))((\hb_2))$ constructed in
~\cite{HT}. We then provide an explicit formula for the unique
normalized $\c((\hb_1))((\hb_2))$-linear trace $\phi$ on
$\dk((\hb_1))((\hb_2))$. Using $\phi$ and the Hochschild
$2n-2$-cocycle $\tau_{2n-2}$ from ~\cite{FFS}, we construct a
Hochschild $2n-2$-cocycle $\psi_{2n-2}$ of the formal Dunkl-Weyl
algebra $\W_{n-1}((\hb_1)) \otimes_{\c((\hb_1))}
\dk((\hb_1))((\hb_2))$. We then prove a local Riemann-Roch theorem
for a version of this cocycle.

\subsection{Hochschild homology.}

Let $G$ be a finite subgroup of $\text{Sp}(2n,\c)$. Then, $G$ acts
by automorphisms on the Weyl algebra
$$\A_n((\hb_1)):= \frac{\c\langle x_1,...,x_n,y_1,...,y_n\rangle((\hb_1))}{\langle [x_i,y_i]=\hb_1, \text{ }1 \leq i \leq n;\,\,\,[x_i,y_j]=0,\,\,i \neq j \rangle  } \text{.}$$
Let $a_j$ denote the number of conjugacy classes of elements of $G$
having eigenvalue $1$ with multiplicity $j$. Let $\A_n^G((\hbar_1))$
denote the subalgebra of elements in $\A_n((\hb_1))$ fixed by $G$.
One has the following result.

\begin{theorem} \la{one} (~\cite{AFLS})
$$ \hh_j(\A_n^G((\hb_1))) \cong \hh^{2n-j}(\A_n^G((\hb_1)),\A_n^G((\hb_1))) \cong \c((\hb_1))^{a_j} \text{. }$$
\end{theorem}
\noindent{\textbf{Standing assumption and notations.}} We recall that when the triple $(\c^{2n},\sum_{i=1}^n dx_i \wedge dy_i,G)$ is indecomposable in the sense of ~\cite{EG}, $a_j=0$ for odd $j$ (see equation 2.12 in ~\cite{EG}). We shall henceforth assume that this is the case. For notational brevity, we shall henceforth denote $\A_n^G((\hb_1))$ by $\A_n^G$, keeping in mind that $\A_n^G$ is an algebra over $\c((\hb_1))$.\\

 In particular, $\hh^2(\A_n^G,\A_n^G)$ is a vector space of dimension $a_{2n-2}$ over $\c((\hb_1))$. The number $a_{2n-2}$ is the number of conjugacy classes of {\it symplectic reflections} in $G$. This is nonzero in general. In other words, $\A_n^G$ has nontrivial formal deformations in general.  Let $\theta \in \hh^2(\A_n^G,\A_n^G)$. Then, $\theta$ defines a $\c((\hbar_1))((\hbar_2))$-linear star product $\star$ on $\A_n^G((\hb_2))$ such that
 $$a \star b = ab+\hb_2(....) \text{ for } a,b \in \A_n^G $$ by Theorem 1.3 of ~\cite{EG}.
 Denote the algebra $(\A_n^G((\hb_2)),\star)$ by $B^{\theta}((\hb_2))$. One has the following generalization of a part of Theorem ~\ref{one}.
 \begin{prop} \la{p1}
 $$ \hh_j(B^{\theta}((\hb_2))) \cong \c((\hb_1))((\hb_2))^{a_j} \cong \hh^{2n-j}(B^{\theta}((\hb_2)),B^{\theta}((\hb_2)))\text{.}$$
 \end{prop}
\begin{proof}
We note that $B^{\theta}[[\hb_2]]:=(\A^G_n[[\hb_2]],\star)$ is a subalgebra of $B^{\theta}((\hb_2))$. Further, as modules over $\c((\hb_1))[[\hb_2]]$, $$\text{C}_{\bullet}(B^{\theta}((\hb_2)))=\varinjlim_{p \geq 0} \hb_2^{-p} \text{C}_{\bullet}(B^{\theta}[[\hb_2]]) \text{. }$$
Therefore, to show that $\hh_j(B^{\theta}((\hb_2))) \cong \c((\hb_1))((\hb_2))^{a_j}$, it suffices to show that
\begin{equation} \la{a1} \hh_j(B^{\theta}[[\hb_2]]) \cong \c((\hb_1))[[\hb_2]]^{a_j} \text{.}\end{equation}
For this, put
$F_p\text{C}_{\bullet}(B^{\theta}[[\hb_2]]):=\hb_2^{-p}\text{C}_{\bullet}(B^{\theta}[[\hb_2]])$
for $p \leq 0$. This is a complete, increasing exhaustive filtration
on $\text{C}_{\bullet}(B^{\theta}[[\hb_2]])$. The completeness of
this filtration follows from the facts that
$\c[[\hb_2]]=\varprojlim_{k \geq 0}
\frac{\c[[\hb_2]]}{\hb_2^k\c[[\hb_2]]}$ and that
$\text{C}_{\bullet}(B^{\theta}[[\hb_2]])$ is a complex of flat
$\c[[\hb_2]]$-modules. This filtration yields a spectral sequence
such that
$$E^{1}_{pq}=\text{H}_{p+q}(\frac{F_p\text{C}_{\bullet}}{F_{p-1}\text{C}_{\bullet}})
\cong \hb_2^{-p}\c((\hb_1))^{a_{p+q}} \text{ for } p \leq 0 $$
with $E^{1}_{pq}=0$ otherwise. The isomorphism above is a consequence of Theorem ~\ref{one}. Clearly, this spectral sequence is bounded above. Since $a_{p+q}=0$ for $p+q$-odd, $E^{1}_{pq}=E^{\infty}_{pq}$ in this case. By the complete convergence theorem, this spectral sequence converges to $\hh_{p+q}(B^{\theta}[[\hb_2]])$. From this, equation ~\eqref{a1} is immediate, proving that $\hh_j(B^{\theta}((\hb_2))) \cong \c((\hb_1))((\hb_2))^{a_j}$.\\

Since the same proof of Theorem \ref{four} applies to this special
case to compute the Hochschild cohomology of $B^\theta((\hbar_2))$
from its Hochschild homology, we leave the details of this proof to
our readers.
\end{proof}

\noindent{\bf Remark.} Alternatively, by suitably modifying the proof of
Proposition 6 of \cite{DE}, we can prove that $B^\theta((\hbar_2))$
is in the class of $VB(2n)$ (See Proposition 2 of \cite{DE}), then
we can obtain the Hochschild cohomology of $B^\theta((\hbar_2))$
from its Hochschild homology using the Van den Bergh duality
\cite{Vdb1}-\cite{Vdb2}.\\

Let $n=1$, $G=\Z_2$ with the generator $\gamma$ of $G$ acting on
$\c\langle x,y\rangle$ by multiplication by $-1$. Consider the
formal symplectic reflection algebra $$\mathcal R:= \frac{\c\langle
x,y \rangle ((\hb_1))((\hb_2)) \rtimes \Z_2}{\langle
[y,x]=\frac{\hb_1(1+2\hb_2\gamma)}{2} \rangle} \text{.}$$ Recall
that there is a vector space isomorphism between $\c[x,y] \rtimes
{\Z_2}((\hb_1))((\hb_2))$ and $\mathcal R$ identifying the
commutative monomial $x^ay^b$ with the image in $\mathcal R$ of the
(noncommutative) monomial $x^ay^b$. Let
$\textbf{e}:=\frac{1+\gamma}{2}$ and let
$$ \mathcal D_2((\hb_1))((\hb_2)):= \c[x,y]^{\Z_2}((\hb_1))((\hb_2))\text{. }$$
Note that $\mathcal D_2((\hb_1))((\hb_2))$ is isomorphic (as
$\c((\hb_1))((\hb_2))$-modules) to the spherical subalgebra
$\textbf{e}\mathcal R\textbf{e}$ of $\mathcal R$. Explicitly, this
isomorphism identifies $f \in \mathcal D_2((\hb_1))((\hb_2))$ with
$f.\textbf{e} \in \textbf{e}\mathcal R\textbf{e}$. As a result, the
product on $\mathcal R$ induces a star product on $\mathcal
D_2((\hb_1))((\hb_2))$. It was shown in ~\cite{HT} that the star
product on $\mathcal D_2((\hb_1))((\hb_2))$ induces a star product
on $\text{C}^{\infty}(D_{\epsilon})^{\Z_2}((\hb_1))((\hb_2))$ for
any $0<\epsilon \leq\infty$ (where $D_{\epsilon}$ denotes the open
disc of radius $\epsilon$ centered at the origin in $\R^2$). This
was done by exhibiting an explicit Moyal-type formula for the
product on $\mathcal D_2((\hb_1))((\hb_2))$ (Theorem 3.10 of
~\cite{HT}). Let $\dk((\hb_1))((\hb_2))$ denote
$\text{C}^{\infty}(\R^2)^{\Z_2}((\hb_1))((\hb_2))$ equipped with the
star product induced by that on $\mathcal D_2((\hb_1))((\hb_2))$.
The proof of the following proposition is completely analogous to
that of a part of Proposition ~\ref{p1}. It is therefore omitted. More
generally, Proposition ~\ref{p2} holds with $\R^2$ replaced by
$D_{\epsilon}$ in the definition of $\dk((\hb_1))((\hb_2))$. We
remark that $\dk((\hb_1))((\hb_2))$ is a bornological algebra with a
canonical bornology coinciding with both the precompact bornology as well as the von Neumann bornology (see Appendix A of ~\cite{PPTT} for definitions), and that the tensor powers used to define
Hochschild chains of $\dk((\hb_1))((\hb_2))$ are bornological.
\begin{prop} \la{p2}
\begin{eqnarray*}
 \hh_0(\dk((\hb_1))((\hb_2))) \cong \hh_2(\dk(\hb_1))((\hb_2))) \cong \c((\hb_1))((\hb_2))\\
 \hh_i(\dk((\hb_1))((\hb_2))) =0 \text{ for } i \neq 0,2 \\
 \end{eqnarray*}
\end{prop}

\subsection{The trace formula.}

Proposition ~\ref{p2} shows that $\dk((\hb_1))((\hb_2))$ has, upto scalar, a unique trace. For compatibility of our notation with that of ~\cite{HT}, we put $z:=x+iy$, $\bar{z}:=x-iy$ where $i:=\sqrt{-1}$. In this subsection, we prove the following explicit formula for a trace on $\dk((\hb_1))((\hb_2))$. For a real number $r$, let $\lfloor r \rfloor$ denote the greatest integer $\leq r$.

\begin{theorem} \la{two}
$[1] \neq 0$ in $\hh_0(\dk((\hb_1))((\hb_2)))$. Let $\phi$ denote the unique $\c((\hb_1))((\hb_2))$-linear trace on $\dk((\hb_1))((\hb_2))$ such that $\phi(1)=1$. Then,
\begin{equation} \la{a2} \phi(f)= \sum_{k=0}^{\infty} \frac{(i\hb_1)^k}{(k!)^2}.\large({\prod}_{l=1}^k{ \large\{\frac{l}{2}+{(-1)}^{l+1}\frac{2\lfloor\frac{l+1}{2}\rfloor\hb_2}{l+1}\large\}}\large).\frac{\partial^{2k}f}{\partial z^k \partial \bar{z}^k}(0,0)\end{equation} for any $f(z,\bar{z}) \in \text{C}^{\infty}(\R^2)^{\Z_2} \subset \dk((\hb_1))((\hb_2))$.
\end{theorem}

Before proving Theorem~\ref{two}, we remark that equation ~\eqref{a2} can be interpreted as a character formula for the spherical subalgebra $\mathcal D_2$ of the formal symplectic reflection algebra $\mathcal R$. It would, for instance, be interesting to have analogous formulas for spherical subalgebras of arbitrary formal symplectic reflection algebras. It would then be interesting to compare generalizations of ~\eqref{a2} for formal Cherednik algebras with character formulas in ~\cite{BEG} and ~\cite{BEG1}.

\begin{proof}
We first show ~\eqref{a2} for any polynomial $f(z,\bar{z})$ in $\mathcal D_2((\hb_1))((\hb_2))$.  We shall perform our computations in $\mathcal R$, keeping in mind that $\mathcal D_2((\hb_1))((\hb_2))$ is the spherical subalgebra $\textbf{e}.\mathcal R.\textbf{e}$ of $\mathcal R$ ($\textbf{e}:=\frac{1+\gamma}{2}$). To begin with,
\begin{equation} \la{a3} [z,\bar{z}]=i\hb_1(1+2\hb_2\gamma) \text{ in } \mathcal R \text{.} \end{equation}
It follows that
\begin{eqnarray}
  \nonumber [\bar{z},z^2]=-i\hb_1(1+2\hb_2\gamma)z-i\hb_1.z(1+2\hb_2\gamma)= -2i\hb_1z-2i\hb_1\hb_2(\gamma.z+z.\gamma)=-2i\hb_1z \\
  \la{a4} \implies [z^2,\bar{z}]=2i\hb_1z \text{ in } \mathcal R \text{.}
\end{eqnarray}
It also follows that
\begin{eqnarray}
\nonumber [z,\bar{z}^q]=\sum_{k=1}^q \bar{z}^{k-1}i\hb_1(1+2\hb_2\gamma)\bar{z}^{q-k}\\
\nonumber = qi\hb_1\bar{z}^{q-1}+2i\hb_1\hb_2\bar{z}^{q-1} \sum_{k=1}^{q} {(-1)}^{q-k}.\gamma\\
\la{a5} = qi\hb_1\bar{z}^{q-1}+i\hb_1\hb_2\bar{z}^{q-1}(1-{(-1)}^q)\gamma \text{ in } \mathcal R \text{.}
\end{eqnarray}
Therefore,
\begin{eqnarray}
\nonumber [z^2,\bar{z}^p]=\sum_{k=1}^p \bar{z}^{k-1}.2i\hb_1z.\bar{z}^{p-k}\\
\nonumber \implies [\frac{z^2}{2i\hb_1},\bar{z}^p]= \sum_{k=1}^p \bar{z}^{k-1}.z.\bar{z}^{p-k} = pz\bar{z}^{p-1}-\sum_{k=1}^p [z,\bar{z}^{k-1}].\bar{z}^{p-k}\\
\nonumber = pz\bar{z}^{p-1} - \sum_{k=1}^p (k-1)i\hb_1\bar{z}^{p-2}-\sum_{k=1}^p i\hb_1\hb_2\bar{z}^{p-2}(1-{(-1)}^{k-1}).{(-1)}^{p-k}\gamma \\
\nonumber = pz\bar{z}^{p-1}-\frac{p(p-1)}{2}i\hb_1\bar{z}^{p-2}-i\hb_1\hb_2\bar{z}^{p-2}\gamma.\sum_{k=1}^p (1-{(-1)}^{k-1}).{(-1)}^{p-k}\\
\la{a6} =pz\bar{z}^{p-1}-\frac{p(p-1)}{2}i\hb_1\bar{z}^{p-2}-i\hb_1\hb_22{(-1)}^p\lfloor \frac{p}{2} \rfloor \bar{z}^{p-2}\gamma \text{ in } \mathcal R \text{.}
\end{eqnarray}

Hence,
\begin{eqnarray}
\nonumber [\frac{z^2}{2i\hb_1},z^{p-2}\bar{z}^p]=z^{p-2}[\frac{z^2}{2i\hb_1},\bar{z}^p]\\
\la{a7} = pz^{p-1}\bar{z}^{p-1}-\frac{p(p-1)}{2}i\hb_1z^{p-2}\bar{z}^{p-2}-i\hb_1\hb_22{(-1)}^p\lfloor \frac{p}{2} \rfloor z^{p-2}\bar{z}^{p-2}\gamma \text{ in } \mathcal R \text{.}
\end{eqnarray}
Note that $z^2$ and $z^{p-2}\bar{z}^{p}$ are elements of $\mathcal D_2((\hb_1))((\hb_2))$ for any $p \geq 2$. It follows from this and \eqref{a7} (with $p=k+1$) that
\begin{equation} \la{a8} [z^k\bar{z}^k]=[i\hb_1\frac{k}{2}z^{k-1}\bar{z}^{k-1} +\frac{2i\hb_1\hb_2{(-1)}^{k+1}\lfloor \frac{k+1}{2} \rfloor}{k+1}z^{k-1}\bar{z}^{k-1}\gamma] \text{ in } \hh_0(\mathcal D_2((\hb_1))((\hb_2))) \text{. }\end{equation}
Note that for $f,g \in \mathcal D_2((\hb_1))((\hb_2))$,
$$[f,g]_{\mathcal D_2((\hb_1))((\hb_2))}=[f.\textbf{e},g.\textbf{e}]_{\mathcal R}= [f,g]_{\mathcal R}.\textbf{e} \text{.} $$
Since $\gamma.\textbf{e}=\textbf{e}$, and each $f \in \mathcal D_2((\hb_1))((\hb_2))$ is identified with $f.\textbf{e} \in \mathcal R$,
\begin{equation} \la{a9} [z^k\bar{z}^k]=i\hb_1(\frac{k}{2}+{(-1)}^{k+1}\frac{2\hb_2\lfloor \frac{k+1}{2} \rfloor}{k+1})[z^{k-1}\bar{z}^{k-1}] \text{ in } \hh_0(\mathcal D_2((\hb_1))((\hb_2))) \text{ for any } k \geq 1 \text{. }\end{equation}
Now recall from ~\cite{HT} that \begin{equation} \la{aht} [z\bar{z},g]=i\hb_1(z\frac{\partial g}{\partial z} -\bar{z} \frac{\partial g}{\partial \bar{z}}) \text{.}\end{equation} It follows from this that $[z\bar{z},z^p\bar{z}^q]=i\hb_1(p-q)z^p\bar{z}^q$. Hence, for $p \neq q$, $[z^p\bar{z}^q]=0$ in $\hh_0(\mathcal D_2((\hb_1))((\hb_2)))$. Hence, for $f \in \c[z,\bar{z}]^{\Z_2} \subset \mathcal D_2((\hb_1))((\hb_2))$,
\begin{equation} \la{a10}
[f]=\sum_{k=0}^{\infty} \frac{1}{(k!)^2}[z^k\bar{z}^k]
\frac{\partial^{2k}f}{\partial z^k \partial \bar{z}^k}(0,0) \text{
in } \hh_0(\mathcal D_2((\hb_1))((\hb_2))) \text{.}
\end{equation}
Equation ~\eqref{a2} follows from ~\eqref{a10} and ~\eqref{a9} for
$f \in \c[z,\bar{z}]^{\Z_2}$. For arbitrary $f \in
\text{C}^{\infty}(\R^2)^{\Z_2} \subset \dk((\hb_1))((\hb_2))$,
~\eqref{a2} clearly makes sense. To complete the proof of the
desired theorem, we only need to verify that for arbitrary $f,g \in
\text{C}^{\infty}(\R^2)^{\Z_2} \subset \dk((\hb_1))((\hb_2))$,
$\phi([f,g])=0$ where $\phi$ is given by  ~\eqref{a2}. Let $k$ be an
arbitrary positive integer. By Theorem 3.10 of ~\cite{HT} and
~\eqref{a2}, the coefficients of $\hbar_1^l$ in $\phi([f,g])$ for $l
\leq k$ depend only on the jets of $f$ and $g$ at $(0,0)$ of order $
\leq p(k)$ for some function $p:\mathbb N \rar \mathbb N$. Since
$\phi$ is indeed a trace on  $\mathcal D_2((\hb_1))((\hb_2))$ and
since there exists a polynomial $\tilde{f} \in \c[z,\bar{z}]^{\Z_2}$
(resp. $\tilde{g} \in \c[z,\bar{z}]^{\Z_2}$) whose jets upto order
$\leq p(k)$ at $(0,0)$ coincide with those of $f$ (resp. $g$), the
coefficients of $\hbar_1^l$ in $\phi([f,g])$ for $l \leq k$ vanish.
This shows that $\phi$ extends to a trace on
$\dk((\hb_1))((\hb_2))$.
\end{proof}

\noindent{\textbf{Remark.}} Theorem ~\ref{two} holds with $\R^2$
replaced by $D_{\epsilon}$ in the definition of
$\dk((\hb_1))((\hb_2))$.

\subsection{A Hochschild $2n-2$-cocycle of $\W_{n-1}((\hbar_1)) \otimes_{\c((\hb_1))} \dk((\hb_1))((\hb_2))$.} Let $\W_{n-1}((\hb_1))$ denote the Weyl algebra
$$ \frac{\c\langle p_1,..,p_{n-1},q_1,....,q_{n-1}\rangle((\hb_1))}{\langle [p_i,q_j]=\hb_1\delta_{ij} \text{ for } 1 \leq i,j \leq n-1 \rangle}\text{.}$$
 Recall from ~\cite{FFS} that $\mathfrak{sp}_{2n-2}$ acts on $\W_{n-1}((\hb_1))$ by derivations. Introduce coordinates $w_1,..,w_{2n-2}$ with $w_{2i-1}=p_i$ and $w_{2i}=q_i$. Suppose that $\omega^0=\sum_i dq_i \wedge dp_i = \frac{1}{2} \sum_{i,j} \omega^0_{ij} dw_i \wedge dw_j$. Then $A=(a^i_j) \in \mathfrak{sp}_{2n-2}$ if and only if  the matrix $(a_{ij})$ is symmetric where $a_{ij}:=\sum_{k=1}^{2n-2} \omega^0_{ik}a^k_j$. The derivation of $\W_{n-1}((\hb_1))$ corresponding to a matrix $A$ is given by
 $$f \mapsto [{\tilde{A}},f]_{\hb_1} \text{ where } \tilde{A}=\frac{1}{2} \sum_{i,j} a_{ij}w_iw_j \text{.}$$
 Similarly, $U(1)$ acts on $\dk((\hb_1))((\hb_2))$ via automorphisms. Explicitly, $$\text{exp}(i\theta)(f)(z,\bar{z}) = f(z\text{exp}(i\theta),\bar{z}\text{exp}(-i\theta)) \text{.}$$
 \begin{prop} \la{p3}
 The derivative of the above $U(1)$-action maps $1 \in \mathfrak{u}_{1}$ to the derivation
 $[{z\bar{z}},-]_{\hb_1}$ of $\dk((\hb_1))((\hb_2))$.
 \end{prop}
 \begin{proof} One notes that for the curve $\theta(t)=t$, $$\frac{d}{dt}(\text{exp}(it)(f))|_{t=0}=iz\frac{\partial f}{\partial z}-i\bar{z}\frac{\partial f}{\partial \bar{z}}=[{z\bar{z}},f]_{\hb_1} \text{.}$$
 The first equality above is by the chain rule and the second equality above is by equation~\eqref{aht}.
 \end{proof}

 Equip $\W_{n-1}((\hb_1)) \otimes_{\c((\hb_1))} \dk((\hb_1))((\hb_2))$ with the Lie bracket $[\text{ },\text{ }]_{\hb_1}$. It follows that one has a map of Lie algebras
 \begin{eqnarray}
  \nonumber \mathfrak{sp}_{2n-2} \oplus \mathfrak{u}_{1} \hookrightarrow \W_{n-1}((\hb_1)) \otimes_{\c((\hb_1))} \dk((\hb_1))((\hb_2))\\
  \la{a11} A \oplus \alpha \mapsto \frac{1}{2} \sum_{i,j} a_{ij}w_i w_j \otimes 1 +1 \otimes {\alpha z \bar{z}} \text{.}
 \end{eqnarray}
Further recall that a $\mathfrak{sp}_{2n-2}$-basic reduced
Hochschild cocycle $\tau_{2n-2}$ of $\W_{n-1}((\hb_1))$ is
constructed in \cite{FFS}. Let $\psi_{2n-2}$ be the unique
Hochschild $2n-2$-cocycle of $\W_{n-1}((\hb_1))
\otimes_{\c((\hb_1))} \dk((\hb_1))((\hb_2))$ such that
  \begin{equation} \la{defpsi} \psi_{2n-2}(a_0 \otimes b_0,...,a_{2n-2} \otimes b_{2n-2})= \tau_{2n-2}(a_0,...,a_{2n-2}).\phi(b_0 \star ...\star b_{2n-2}) \end{equation}
 for $a_0,...,a_{2n-2} \in \W_{n-1}((\hb_1))$ and $b_0,..,b_{2n-2} \in \dk((\hb_1))((\hb_2))$.
 \begin{prop} \la{p4}
 $\psi_{2n-2}$ is $\mathfrak{sp}_{2n-2} \oplus \mathfrak{u}_{1}$-basic. In other words, for any $\gamma \in \mathfrak{sp}_{2n-2} \oplus \mathfrak{u}_{1} \oplus \c((\hb_1))((\hb_2))$,
 \begin{eqnarray}
 \la{a12} \sum_{j=0}^{2n-2} \psi_{2n-2}(a_0 \otimes b_0,...,[\gamma,a_j \otimes b_j]_{\hb_1},...,a_{2n-2}\otimes b_{2n-2})=0\\
 \la{a13} \sum_{j=1}^{2n-2} \psi_{2n-2}(a_0 \otimes b_0,....,\gamma,a_j \otimes b_j,...,a_{2n-3} \otimes b_{2n-3})=0
 \end{eqnarray}
 for $a_i \in \W_{n-1}((\hb_1))$ and $b_i \in \dk((\hb_1))((\hb_2))$.
 \end{prop}
 \begin{proof}
 We verify this proposition for the case when $\gamma= \mu \otimes 1$ with $\mu \in \mathfrak{sp}_{2n-2}$ and the case when $\gamma= 1 \otimes \nu$ with $\nu \in \mathfrak{u}_{1}$. The case when $\gamma \in \c((\hb_1))((\hb_2))$ is easy to handle and left to the reader.\\

 For $\gamma=\mu \otimes 1$, $[\gamma, a_j \otimes b_j]_{\hb_1}=[\mu,a_j]_{\hb_1} \otimes b_j$.  In this case,
 \begin{eqnarray}
 \nonumber \sum_{j=0}^{2n-2} \psi_{2n-2}(a_0 \otimes b_0,...,[\gamma,a_j \otimes b_j]_{\hb_1},...,a_{2n-2}\otimes b_{2n-2})\\
  \la{a14} =\sum_{j=0}^{2n-2}
 \tau_{2n-2}(a_0,...,[\mu,a_j]_{\hb_1},..,a_{2n-2}).\phi(b_0 \star .... \star b_{2n-2}) = 0 \text{.}\\
 \nonumber \sum_{j=1}^{2n-2} \psi_{2n-2}(a_0 \otimes b_0,....,\gamma,a_j \otimes b_j,...,a_{2n-3} \otimes b_{2n-3})\\
 \la{a15}=\sum_{j=1}^{2n-2}
 \tau_{2n-2}(a_0,...,\mu,a_j,..,a_{2n-3}).\phi(b_0 \star .... \star b_{2n-3}) = 0 \text{.}
 \end{eqnarray}
 The last equalities in ~\eqref{a14} and ~\eqref{a15} are because $\tau_{2n-2}$ is a $\mathfrak{sp}_{2n-2}$-basic cocycle by ~\cite{FFS}.\\

  For $\gamma= 1 \otimes \nu$, $[\gamma, a_j \otimes b_j]_{\hb_1}=a_j \otimes [\nu,b_j]_{\hb_1}$. In this case,
 \begin{eqnarray}
 \nonumber \sum_{j=0}^{2n-2} \psi_{2n-2}(a_0 \otimes b_0,...,[\gamma,a_j \otimes b_j]_{\hb_1},...,a_{2n-2}\otimes b_{2n-2})\\
 \nonumber =\sum_{j=0}^{2n-2} \tau_{2n-2}(a_0,...,a_{2n-2}).\phi(b_0 \star ...\star [\nu,b_j]_{\hb_1} \star...\star b_{n-2})\\
  \la{a16} = \tau_{2n-2}(a_0,...,a_{2n-2}).\phi([\nu,b_0 \star ...\star b_{2n-2}]_{\hb_1})=0 \text{.}\\
  \nonumber \sum_{j=1}^{2n-2} \psi_{2n-2}(a_0 \otimes b_0,....,\gamma,a_j \otimes b_j,...,a_{2n-3} \otimes b_{2n-3})\\
 \la{a17}  = \sum_{j=1}^{2n-2} \tau_{2n-2}(a_0,...,1,a_j,...,a_{2n-3}).\phi(b_0 \star ...\star \nu \star b_j \star ..\star b_{2n-3})=0 \text{.}
 \end{eqnarray}
 The last equality in ~\eqref{a17} is because the cocycle $\tau_{2n-2}$ is normalized (see ~\cite{FFS}).
 \end{proof}

\subsection{A local Riemann-Roch theorem.}
For notational brevity, we use $\K$ to denote $\c((\hb_1))((\hb_2))$
in this section. Fix an $N \gg n$. Put $R:=\W_{n-1}
\otimes_{\c((\hb_1))} \dk(\hb_1,\hb_2))$. Let
$\mathfrak{g}:=\mathfrak{gl}_N(R)$ and note that
$\mathfrak{sp}_{2n-2}(\K)$ and $\mathfrak{u}_{1}(\K)$ are Lie
subalgebras of $\mathfrak{gl}_n(R)$ via ~\eqref{a11} and the
diagonal embedding $R \hookrightarrow \mathfrak{gl}_n(R)$.  Let
$\mathfrak{h}:= \mathfrak{sp}_{2n-2}(\K) \oplus \mathfrak{u}_{1}(\K)
\oplus \mathfrak{gl}_N(\K)$.
By section 3.1 and 3.2 of ~\cite{FFS}, $\psi_{2n-2}$ may be viewed as a $2n-2$-Lie cocycle $\Psi_{2n-2} \in \text{C}^{2n-2}_{\text{Lie}}(\mathfrak{g},\mathfrak{h};\mathfrak{g}^*)$. Let $\text{ev}_1:\text{C}^{2n-2}_{\text{Lie}}(\mathfrak{g},\mathfrak{h};\mathfrak{g}^*) \rar \text{C}^{2n-2}_{\text{Lie}}(\mathfrak{g},\mathfrak{h};\K)$ is evaluation at the identity. We compute $\text{ev}_1\Psi_{2n-2}$. For this purpose, we closely follow Section 5 of ~\cite{PPT}.\\

 Choose coordinates $p_1,..,p_{n-1},q_1,...,q_{n-1}$ of $\R^{2n-2}$ and a coordinate $z$ of $\c^1$. For an element $a \in R$, let $a_i$ denote the homogenous component of degree $i$. Note that $a_2=\sum_{i,j} a_{ij} w_iw_j+\alpha z^2+\beta z\bar{z}+\lambda \bar{z}^2$ where $w_{2i-1}=p_i$ and $w_{2i}=q_i$. Let $a_2':=a_2-\alpha.z^2-\lambda \bar{z}^2$. Then, one has a $\mathfrak{h}$-equivariant projection
 \begin{eqnarray}
 \nonumber \text{pr}:\mathfrak{g} \rar \mathfrak{h} \\
 \la{1.4.1} M \otimes a \mapsto \frac{\text{tr}(M)}{N}a_2'+Ma_0 \text{.}
 \end{eqnarray}
Let $C \in \text{Hom}_{\K}(\wedge^2\mathfrak{g}, \mathfrak{h})$ be the curvature of $\text{pr}$ with $$C(v,w)=[\text{pr}(v),\text{pr}(w)]_{\hb_1}-\text{pr}([v,w]_{\hb_1}) \text{.}$$
One then has a Chern-Weil homomorphism $\chi: (S^{\bullet}\mathfrak{h}^*)^{\mathfrak{h}} \rar \text{C}^{2\bullet}(\mathfrak{g},\mathfrak{h};\K)$ defined by
\begin{eqnarray}
  \nonumber \chi(P)(v_1\wedge...\wedge v_{2k})=\frac{1}{k!} \sum_{\stackrel{\sigma \in S_{2k}}{\sigma(2i-1)<\sigma(2i)}} {(-1)}^{\sigma}P(C(v_{\sigma(1)},v_{\sigma(2)}),....,C(v_{\sigma(2k-1)},v_{\sigma(2k)})) \text{.}
      \end{eqnarray}
\begin{prop} \la{p1.4.1} For $N>>n$, $\chi: (S^{q}\mathfrak{h}^*)^{\mathfrak{h}} \rar \text{H}^{2q}(\mathfrak{g},\mathfrak{h};\K)$ is an isomorphism for $q \leq n-1$.
\end{prop}
\begin{proof} The proof is almost identical to that of Proposition 5.2 in ~\cite{FFS}. The only (minor) difference is that for $p,q \geq 0$, $$\text{H}^p(\mathfrak{gl}_N(R),S^q\mathfrak{gl}_N(R)^*) = \K \text{ when } p=2n-2 \text{ and } 0 \text{ otherwise.}$$
\end{proof}
Form the subspace $$W_{n-1,1,N}:=\{\sum_i f_ip_i \otimes 1+ gz\bar{z} \otimes 1 +\sum_k h_k \otimes M_k \text{  }|\text{  } f_i,g,h_k \in \K[q_1,..,q_{n-1},z^2]\} \text{.}$$
\begin{prop}
$W_{n-1,1,N}$ is a Lie-subalgebra of $\mathfrak{gl}_n(R)$.
\end{prop}
\begin{proof}
$W_{n-1,1,N}$ is clearly a subspace of $\mathfrak{gl}_n(R)$. The
rest follows from direct computations. For example, if
$g_1=\tilde{g}_1(q_1,...,q_{n-1})z^{2k_1}$ and if
$g_2=\tilde{g}_2(q_1,...,q_{n-1})z^{2k_2}$, then
\begin{eqnarray*}
[g_1z\bar{z} \otimes 1,g_2z\bar{z} \otimes 1]_{\hb_1} = \tilde{g}_1\tilde{g}_2[z^{2k_1}z\bar{z},z^{2k_2}z\bar{z}]_{\hb_1} \otimes 1 \text{.}
\end{eqnarray*}
Since
\begin{eqnarray*}
[z^{2k_1}z\bar{z},z^{2k_2}z\bar{z}]_{\hb_1}=z^{2k_1}[z\bar{z},z^{2k_2}]_{\hb_1}z\bar{z}-z^{2k_2}[z\bar{z},z^{2k_2}]_{\hb_1}z\bar{z}=i\hb_1(2k_2-2k_1)z^{2k_1+2k_2}z\bar{z}\text{ ,} \\
\text{  }[g_1z\bar{z}\otimes 1,g_2z\bar{z}\otimes 1]_{\hb_1}=2i(k_2-k_1)\tilde{g}_1\tilde{g}_2z^{2k_1+2k_2}z\bar{z} \otimes 1 \in W_{n-1,1,N} \text{.}
\end{eqnarray*}
Note that in the above computation, we have used equation ~\eqref{aht}.
\end{proof}

Put $\mathfrak{h}_1= W_{n-1,1,N} \cap \mathfrak{h}$. Then,
$\mathfrak{h}_1=\mathfrak{gl}_{n-1}(\K) \oplus \mathfrak{u}_{1}(\K)
\oplus \mathfrak{gl}_N(\K)$. Denote by $\widetilde{W}_{n-1,1,N}$ the
Lie subalgebra of $W_{n-1,1,N}$ of elements of the form $\sum_i
f_ip_i \otimes 1+gz\bar{z} \otimes 1+\sum_k h_k \otimes M_k$ where
$f_i,g,h_k \in \K[q_1,...,q_{n-1}]$.

\begin{prop} \la{p1.4.2}
The Chern-Weil homomorphism
$\chi:(S^q{\mathfrak{h}_1}^*)^{\mathfrak{h}_1} \rar
\text{H}^{2q}(W_{n-1,1,N},\mathfrak{h}_1;\K)$ is injective, for $q
\leq n-1$. Further, $\text{H}^{2q}(W_{n-1,1,N},\mathfrak{h}_1;\K)
\cong \text{C}^{2q}(\widetilde{W}_{n-1,1,N}, \mathfrak{h}_1;\K)$.
\end{prop}

\begin{proof}
The proof of Proposition 5.2 in ~\cite{PPT} goes through in this case as well with minor changes which we shall point out. $\gamma=\text{Id}$ and $W^\gamma_{k,n-k,N}$ is replaced by $W_{n-1,1,N}$. With notation paralleling that in the proof of Proposition 5.2 on ~\cite{PPT}, the basis of $W_{n-1,1,N}$ comprises of $q^{\alpha_j}z^{2\beta_j}p_j \otimes 1$, $q^{\alpha}z^{2\beta}z\bar{z} \otimes 1$ and $q^{\alpha_{st}}z^{2\beta_{st}} \otimes E_{st}$ where $\alpha,\alpha_j$ and $\alpha_{st}$ are multiindices, $\beta_j,\beta,\beta_{st} \in \mathbb N \cup\{0\}$ and $E_{st}$ denotes an elementary matrix. As in ~\cite{PPT}, we study the $\mathfrak{h}_1$-action on $\text{Hom}_{\K}(\wedge^{\bullet}(W_{n-1,1,N}),\K)$. The elements $\sigma_1:=\sum_{j=1}^{n-1}q_jp_j \otimes 1$ and $\sigma_2:=z\bar{z} \otimes 1$ act diagonally in this case as well. The formula for the $\sigma_1$-action is unchanged from that in ~\cite{PPT}. That for the $\sigma_2$-action is changed from the corresponding formula in ~\cite{PPT} by a factor of $i:=\sqrt{-1}$. We present the formulas here in order to be explicit.
\begin{eqnarray*}
 \text{ } [\sigma_1,q^{\alpha_j}z^{2\beta_j}p_j \otimes 1]_{\hb_1} = ((\sum_{l=1}^{n-1}\alpha_j^l)-1)q^{\alpha_j}z^{2\beta_j}p_j \otimes 1\\
\text{  }[\sigma_1,q^{\alpha}z^{2\beta}z\bar{z} \otimes 1]_{\hb_1}= (\sum_{l=1}^{n-1}\alpha^l)q^{\alpha}z^{2\beta}z\bar{z} \otimes 1\\
\text{ }[\sigma_1,q^{\alpha_{st}}z^{2\beta_{st}} \otimes E_{st}]_{\hb_1}= (\sum_{l=1}^{n-1}\alpha_{st}^l)q^{\alpha_{st}}z^{2\beta_{st}} \otimes E_{st}\\
\text{ }[\sigma_2,q^{\alpha_j}z^{2\beta_j}p_j \otimes 1]_{\hb_1}= 2i\beta_jq^{\alpha_j}z^{2\beta_j}p_j \otimes 1\\
\text{ }[\sigma_2,q^{\alpha}z^{2\beta}z\bar{z} \otimes 1]_{\hb_1}=2i\beta q^{\alpha}z^{2\beta}z\bar{z} \otimes 1\\
\text{ }[\sigma_2,q^{\alpha_{st}}z^{2\beta_{st}} \otimes E_{st}]_{\hb_1}=2i\beta_{st}q^{\alpha_{st}}z^{2\beta_{st}} \otimes E_{st}
\end{eqnarray*}
The rest of the proof proceeds as in the proof of Proposition 5.2 of ~\cite{PPT} with the obvious modifications.
\end{proof}
For $X=X_1 \oplus X_2 \oplus X_3 \in \mathfrak{sp}_{2n-2}(\K) \oplus \mathfrak{u}_{1}(\K) \oplus \mathfrak{gl}_N(\K)$, let
$$ \hat{A}_{\hb_1}\text{Ch}_{\phi}\text{Ch}(X):=(\text{det}(\frac{\frac{\hb_1X_1}{2}}{\text{sinh}(\frac{\hb_1X_1}{2})}))^{\frac{1}{2}}\phi(\text{exp}_{\star}(X_2))
\text{tr}(\text{exp}(X_3)) \text{.}$$
Here, $\star$ denotes the star product in $\dk((\hb_1))((\hb_2))$ and $X_2 \in \mathfrak{u}_{1}(\K)$ is viewed as an element of $\dk((\hb_1))((\hb_2))$ via equation ~\eqref{a11}. We are now in a position to state the local Riemann-Roch.
\begin{theorem} \la{five}
$$[\text{ev}_1\Psi_{2n-2}] = {(-1)}^{n-1} \chi((\hat{A}_{\hb_1}\text{Ch}_{\phi}\text{Ch})_{n-1}) \text{.}$$
\end{theorem}

\begin{proof}
By Proposition ~\ref{p1.4.1}, there exists a unique polynomial $Q
\in (S^{n-1}\mathfrak{h}^*)^{\mathfrak{h}}$ such that
$[\text{ev}_1\Psi_{2n-2}]=\chi(Q)$. Since $Q$ is
$\mathfrak{h}$-invariant, it is determined by its value on the
Cartan subalgebra $\mathfrak{a}$ of $\mathfrak{h}$. $\mathfrak{a}$
is spanned by $q_ip_i \otimes 1$ for $1\leq i\leq n-1$, $z\bar{z}
\otimes 1$ and $ 1 \otimes E_{rr}$ for $1 \leq r \leq N$. By
Proposition ~\ref{p1.4.1} and Proposition ~\ref{p1.4.2}, we have the
following commutative diagram, whose vertical arrows are injective
for $k \leq n-1$.
$$\begin{CD}
(S^k\mathfrak{h}^*)^{\mathfrak{h}} @>>> (S^k\mathfrak{h}_1^*)^{\mathfrak{h}_1} @>>\text{id}> (S^k\mathfrak{h}_1^*)^{\mathfrak{h}_1}\\
   @VV{\chi}V                             @VV{\chi}V                                @VV{\chi}V\\
\text{H}^{2k}(\mathfrak{g},\mathfrak{h};\K) @>>> \text{H}^{2k}(W_{n-1,1,N},\mathfrak{h}_1;\K) @>\cong>> \text{C}^{2k}(\tilde{W}_{n-1,1,N},\mathfrak{h}_1;\K)\\
\end{CD}$$
Further, $\mathfrak{h}$ and $\mathfrak{h}_1$ have the same Cartan subalgebra $\mathfrak{a}$. Since invariant polynomials are determined by their values on $\mathfrak{a}$, the horizontal arrow in the top left of the above diagram is an injection. It follows that the horizontal arrow in the bottom left of the above diagram is an injection as well. Therefore, in order to determine $Q$, one may work with the restriction of $\mathfrak{g}$ to $\widetilde{W}_{n-1,1,N}$. The identity required to prove the desired theorem then becomes an identity of cocycles rather than an identity at the level of cohomology. Hence, $\text{ev}_1\Psi_{2n-2}=\chi(Q)$ in $\text{C}^{2n-2}(\tilde{W}_{n-1,1,N},\mathfrak{h}_1;\K)$.\\

We now show that $Q={(-1)}^{n-1}(\hat{A}_{\hb_1}\text{Ch}_{\phi}\text{Ch})_{n-1}$ as follows. Put
$$u_{ij}:=-\frac{1}{2}q_i^2p_i\delta_{ij}+q_iq_jp_j,\,\,\,\,\,\, v_{ir}:=q_i \otimes E_{rr}\,\,\,\,\,\, w_i:=q_iz\bar{z} \text{.}$$
Note that
$$[p_i,u_{ij}]_{\hb_1}=p_jq_j \,\,\,\,\,\,[p_i,v_{ir}]_{\hb_1}=E_{rr}\,\,\,\,\,\,[p_i,w_i]_{\hb_1}=z\bar{z} \text{.}$$
Further, $\text{pr}(u_{ij})=\text{pr}(v_{ir})=\text{pr}(w_i)=0$. As in the proof of Lemma 5.3 of ~\cite{FFS}, it follows that
 \begin{equation} \la{1.4.3} (\text{ev}_1\Psi_{2n-2})(p_1\wedge x_1 ......\wedge p_{n-1} \wedge x_{n-1})= {(-1)}^{n-1}Q(\frac{\partial x_1}{\partial q_1},...,\frac{\partial x_{n-1}}{\partial q_{n-1}})\end{equation}
for any $x_1,..,x_{n-1}$ of the form $u_{ij}$ for $j \leq
i$,$v_{ir}$ or $w_i$. As in the proof of Lemma 5.3 in ~\cite{FFS},
equation ~\eqref{1.4.3} uniquely determines the polynomial $Q$.
Proceeding as in the proof of Lemma 5.3 of ~\cite{FFS}, one shows
that $Q$ coincides with the polynomial whose restriction to
$\mathfrak{a}$ satisfies
\begin{eqnarray}
\nonumber Q(M_1 \otimes a_1 \otimes b_1,...,M_{n-1} \otimes a_{n-1} \otimes b_{n-1})= \text{tr}(M_1...M_{n-1})\phi(b_1 \star ...\star b_{n-1})\\
 \la{1.4.4} \mu_{n-1}\int_{[0,1]^{n-1}} \prod_{ 1 \leq i \leq j \leq n-1}
 \text{e}^{\hb_1\psi(u_i-u_j)\alpha_{ij}}\times(a_1 \otimes...\otimes a_{n-1})du_1...du_{n-1} \text{.}\end{eqnarray}
Here, $\mu_{n-1}$ and $\psi$ are exactly as in Section 2.3 of ~\cite{FFS}. Proceeding exactly as in the proof of Theorem 5.3 of ~\cite{PPT}, one then shows that $Q(Y+Z)=(\hat{A}_{\hb_1}\text{Ch})(Y).\text{Ch}_{\phi}(Z)$ for $Y=\sum_i \nu_iq_ip_i+\sum_r\sigma_rE_{rr}$ and $Z=\tau z\bar{z}$ with $\nu_i,\sigma_r,\tau \in \K$. This proves the desired theorem.
\end{proof}

\noindent{\textbf{Remark.}} We note that every element $X_2 \in
\mathfrak{u}_1(\K)$ is of the form $\tau z \bar{z}$ for some $\tau
\in \K$. Applying Theorem~\ref{two}, one immediately obtains the
following formula for $\text{Ch}_{\phi}(X_2)$.
\begin{equation} \la{genGtr} \text{Ch}_{\phi}(X_2)= \sum_{k=0}^{\infty} \frac{(i \hb_1 \tau)^k}{k!}.\large\prod_{l=1}^{l=k} (\frac{l}{2}+{(-1)}^{l+1}\frac{2\lfloor \frac{l+1}{2} \rfloor \hb_2}{l+1}) \text{.} \end{equation}

\section{Hochschild homology of ${\mathfrak{A}}_{M/\Z_2}((\hb_1))((\hb_2))$.}

Let $M$ be a smooth, compact symplectic manifold with a symplectic $\Z_2$-action. Let $\omega$ denote the symplectic form on $M$. Let $M_2^{\gamma}$-denote a collection of codimension $2$ components of the submanifold $M^{\gamma}$ of points fixed by the $\Z_2$-action. In ~\cite{HT}, Halbout and Tang construct a quantization $\mathfrak{A}_{M/Z_2}((\hb_1))((\hb_2))$ of
 $\A_{M/\Z_2}((\hb_1))$ where $\A_{M/\Z_2}((\hb_1))$ is the $\Z_2$-invariant subalgebra of a quantization of $\text{C}^{\infty}(M)$ with characteristic class $\omega$. In this section, we compute $\hh_{\bullet}(\mathfrak{A}_{M/\Z_2}((\hb_1))((\hb_2)))$ as well as the Hochschild cohomology $\hh^{\bullet}(\mathfrak{A}_{M/\Z_2}((\hb_1))((\hb_2)),\mathfrak{A}_{M/\Z_2}((\hb_1))((\hb_2)))$. Starting from this section, we consider $\text{C}^{\infty}(M)$, $\A_{M/\Z_2}((\hb_1))$, and $\mathfrak{A}_{M/Z_2}((\hb_1))((\hb_2))$ as bornological algebras with their canonical bornologies (\cite{PPTT}, A.6.).
 \subsection{Construction of $\mathfrak{A}_{M/\Z_2}((\hb_1))((\hb_2))$.} To start with, we reformulate the construction $\mathfrak{A}_{M/\Z_2}((\hb_1))((\hb_2))$ of ~\cite{HT} in a way that makes it convenient to compute Hochschild homology. $\mathfrak{A}_{M/\Z_2}((\hb_1))((\hb_2))$ is constructed in three steps. Before proceeding with these steps, we fix a $\Z_2$-invariant almost complex structure on $M$. This also defines a $\Z_2$-invariant metric on $M$. Fix any sufficiently small positive real number $\epsilon$. Let $B_{\epsilon}$ denote an $\epsilon$-tubular neighborhood of $M_2^{\gamma}$. Let $N$ denote the normal bundle to $M_2^{\gamma}$. Let $N_{\epsilon}$ denote the $\epsilon$-neighborhood of the zero section of $N$. We require that $\epsilon$ be small enough for us to be able to identify $B_{\epsilon}$ with $N_{\epsilon}$ via the exponential map with respect to some $\Z_2$ invariant metric on $N$. We fix this exponential map, which we denote by $\text{exp}$.

 \subsubsection{Step 1: near $M^{\gamma}$.}
 {\it We first explain the construction in the vicinity of $M_2^{\gamma}$.}
 Let $P$ denote the principal $U(1)$-bundle of Hermitian frames on  $N$. Over $M_2^{\gamma}$, one has the {\it Dunkl-Weyl algebra bundle}
 $$\V:= P \times_{U(1)} \dk((\hb_1))((\hb_2)) \text{.}$$
 Let $J_1^{sp}M_2^{\gamma}$ denote the bundle of symplectic frames on $TM_2^{\gamma}$. One also has  a bundle of algebras
 $$\Ww:= J_1^{sp}M_2^{\gamma} \times_{\text{Sp}(2n-2)} \W_{n-1}((\hb_1))$$ over $M_2^{\gamma}$.
 ~\cite{HT} construct a Fedosov connection $D$ on $\bigwedge^{\bullet} T^*M_2^{\gamma} \otimes \Ww \otimes_{\c((\hb_1))} \V$. This makes $(\bigwedge^{\bullet} T^*M_2^{\gamma} \otimes \Ww \otimes_{\c((\hb_1))} \V,D)$ a sheaf of DG-algebras over $M_2^{\gamma}$. The corresponding sheaf of degree $0$-flat sections is a sheaf of algebras over $M_2^{\gamma}$, which we shall denote by $\Aa_{D}$. We note that for any open subset $U$ of $M_2^{\gamma}$, $\Gamma(U,\Aa_{D}) \cong \text{C}^{\infty}(U,\V)$ as vector spaces. Since $\dk((\hb_1))((\hb_2)) \cong \text{C}^{\infty}(\R^2)^{\Z_2}((\hb_1))((\hb_2))$ as vector spaces,
 $\Gamma(U,\Aa_{D}) \cong \text{C}^{\infty}(U,P \times_{U(1)} \text{C}^{\infty}(\R^2)^{\Z_2}((\hb_1))((\hb_2))) \cong \text{C}^{\infty}(N|_{U})^{\Z_2}((\hb_1))((\hb_2))$  as vector spaces. As a result, $\text{C}^{\infty}(N|_{U})^{\Z_2}((\hb_1))((\hb_2))$ inherits a new product $\star$ from $\Gamma(U,\Aa_{D})$. Let $\C^{\infty}(N)$ be the sheaf on $M_2^{\gamma}$ such that $\Gamma(U,\C^{\infty}(N))=\text{C}^{\infty}(N|_{U})$. The star product $\star$ inherited from $\Aa_{D}$ makes $\C^{\infty}(N)^{\Z_2}((\hb_1))((\hb_2))$ a sheaf of algebras on $M_2^{\gamma}$. \\

 Let $d$ denote the ($\Z_2$-invariant) Riemannian metric on $M$. Consider open subsets on $N$ that are of the form $\text{exp}^{-1}(N_{\alpha,U})$, where
 $$ N_{\alpha,U}:=\{(x \in M  \text{ }|\text{ } d(x,U) < \alpha \}$$
 where $U$ is an open subset of $M_2^{\gamma}$ and $0 < \alpha \leq \epsilon$. In what follows, we shall denote  $\text{exp}^{-1}(N_{\alpha,U})$ by $N_{\alpha,U}$ itself, since it shall be clear from the context whether we are referring to a tubular neighnorhood of $U$ or its image in $N|_{U}$ via the map $\text{exp}^{-1}$. Since the star product on $\dk((\hb_1))((\hb_2))$ is $\gamma$-local, the product on $\text{C}^{\infty}(N|_{U})^{\Z_2}((\hb_1))((\hb_2))$ restricts to one on $\text{C}^{\infty}(N_{\alpha,U})^{\Z_2}((\hb_1))((\hb_2))$. Clearly, the assignment
 $$ N_{\alpha,U} \mapsto (\text{C}^{\infty}(N_{\alpha,U})^{\Z_2}((\hb_1))((\hb_2)),\star) $$
 is a contravariant functor
 $$ \large( \{N_{\alpha,U}\},\subset \large) \rar \text{algebras} \text{.}$$

{ \it Near components of $M^{\gamma}$ that are not in $M_2^{\gamma}$.} On open subsets of $M$ that are of the form $N_{\alpha,U}$ with $\alpha \leq \epsilon$ and $U$ a subset of $M^{\gamma} \setminus M_2^{\gamma}$, we make the assignment
 $$ N_{\alpha,U} \mapsto (\text{C}^{\infty}(N_{\alpha,U})^{\Z_2}((\hb_1))((\hb_2)),\star_F) $$
 where $\star_F$ is the product coming from the Fedosov quantization of $\text{C}^{\infty}(M \setminus M_2^{\gamma})$ with Weyl curvature $\omega$.

 \subsubsection{Step 2: away from $M^{\gamma}$.}
 On $M^{-}:=M \setminus M^{\gamma}$, things are easy. Note that $M^-$ has a free $\Z_2$-action. The Fedosov quantization $\Aa_{M^-}((\hb_1))((\hb_2))$ of $\C^{\infty}(M^-)$ with Weyl curvature $\omega$ yields a sheaf of algebras on $M^-$ with a star product $\star_F$.

 \subsubsection{Step 3: patching.}
 Let $U$ be an open subset of $M_2^{\gamma}$. Let $N^*_{\alpha,U}:=N_{\alpha,U} \setminus \{x \text{ }| \text{ } x \in U\}$. Since the product on $\dk((\hb_1))((\hb_2))$ is $\gamma$-local, one further has restriction maps
 $$(\text{C}^{\infty}(N_{\alpha,U})^{\Z_2}((\hb_1))((\hb_2)),\star) \rar (\text{C}^{\infty}(N^*_{\alpha,U})^{\Z_2}((\hb_1))((\hb_2)),\star) \text{.}$$
 Further, Proposition 4.1 of ~\cite{HT} argues that $(\text{C}^{\infty}(N^*_{\alpha,U})^{\Z_2}((\hb_1))((\hb_2)),\star)$ is isomorphic as algebras to
 $(\Gamma(N^*_{\alpha,U},\Aa_{M^-}((\hb_1))((\hb_2)))^{\Z_2},\star_F)$. Note that one has a natural inclusion of algebras
 $$ (\Gamma(N^*_{\alpha,U},\Aa_{M^-}((\hb_1))((\hb_2)))^{\Z_2},\star_F) \hookrightarrow (\Gamma(N^*_{\alpha,U},\Aa_{M^-}((\hb_1))((\hb_2))),\star_F) \text{.}$$
Further observe that the collection of open subsets $\{U \subset M^- \text{ }|\text{ }U \text{ open} \}\cup \{ N_{\alpha,U} \text{ }|\text{ }U \subset M_2^{\gamma}\\ \text{ is open and } 0<\alpha \leq \epsilon\}$ forms a base for the topology on $M$. We denote this base by $\mathcal B_{\epsilon}$. Consider the assignment
\begin{eqnarray}
 \nonumber U \mapsto (\Gamma(U,\Aa_{M^-}((\hb_1))((\hb_2))),\star_F) \text{ for }U \subset M^- \text{ open}\\
\nonumber N_{\alpha,U} \mapsto (\text{C}^{\infty}(N_{\alpha,U})^{\Z_2}((\hb_1))((\hb_2)),\star) \text{ for } U \subset M_2^{\gamma}\\
 \la{a18} N_{\alpha,U} \mapsto (\text{C}^{\infty}(N_{\alpha,U})^{\Z_2}((\hb_1))((\hb_2)),\star_F) \text{ for } U \subset M^{\gamma} \setminus
 M_2^{\gamma}.
\end{eqnarray}
For an open subset $V$ of $M^-$ and an open subset $U$ of $M_2^{\gamma}$, the restriction map $\rho_{N_{\alpha,U},V}$ is defined via the isomorphism from Proposition 4.1 of ~\cite{HT} mentioned above. In all other cases, the restriction maps are natural.
\begin{prop} \la{p5} The assignment ~\eqref{a18} with restriction maps defined as above gives a $\mathcal B_{\epsilon}$-presheaf of algebras on $M$.
\end{prop}
\begin{proof}
The only nontrivial verification we need to do is to verify compatibility of restriction maps for a sequence of open subsets of $M$ of the form
$W \subset N_{\beta,V} \subset N_{\alpha,U}$ where $V \subset U$ are open subsets of $M_2^{\gamma}$, $\beta \leq \alpha$ and $W$ is an open subset of $M^-$. The required verification follows from the commutativity of the following diagram where the horizontal arrows are the natural restriction maps.
$$\begin{CD}
\text{C}^{\infty}(N_{\alpha,U})^{\Z_2}((\hb_1))((\hb_2)) @>>> \text{C}^{\infty}(N_{\beta,V})^{\Z_2}((\hb_1))((\hb_2))\\
  @VVV                                                                @VVV\\
\text{C}^{\infty}(N^*_{\alpha,U})^{\Z_2}((\hb_1))((\hb_2)) @>>> \text{C}^{\infty}(N^*_{\beta,V})^{\Z_2}((\hb_1))((\hb_2))\\
  @VVV                                                               @VVV\\
  \Gamma(N^*_{\alpha,U},\Aa_{M^-}((\hb_1))((\hb_2)))^{\Z_2} @>>> \Gamma(N^*_{\beta,V},\Aa_{M^-}((\hb_1))((\hb_2)))^{\Z_2}\\
  @VVV                      @VVV\\
   \Gamma(N^*_{\alpha,U},\Aa_{M^-}((\hb_1))((\hb_2))) @>>>  \Gamma(N^*_{\beta,V},\Aa_{M^-}((\hb_1))((\hb_2)))
\end{CD}$$
The vertical arrows in the topmost square of the above diagram are
the natural restriction maps. Those in the middle square are the
isomorphisms from Proposition 4.1 of ~\cite{HT}. Those in the bottom
square are natural inclusions. The bottom square obviously commutes.
The commutativity of the top square follows from $\gamma$-locality
of the product on $\dk((\hb_1))((\hb_2))$. To understand the middle
square, we recall the construction of the isomorphism from
Proposition 4.1 of ~\cite{HT}. This isomorphism is constructed in
two steps. To start with, as pointed out in Remark 4.3 of
~\cite{HT}, one has a $U(1)$-equivariant isomorphism between
$\dk((\hb_1))((\hb_2))|_{D^*_{\epsilon}}$ and
$\W_2^{\Z_2}((\hb_1))((\hb_2))|_{D^*_{\epsilon}}$ compatible with
restriction to $D^*_{\epsilon'}$ for any $\epsilon' < \epsilon$.
This enables the construction of an isomorphism of algebras between
$(\text{C}^{\infty}(N^*_{\alpha,U})^{\Z_2}((\hb_1))((\hb_2)),\star)$
and
$(\text{C}^{\infty}(N^*_{\alpha,U})^{\Z_2}((\hb_1))((\hb_2)),\star_W)$
where the latter algebra is constructed following Section 3.1.1
after replacing $\V$ by $\V_W:=P \times_{U(1)}
\W_2^{\Z_2}((\hb_1))((\hb_2))$. Since the $U(1)$-equivariant
isomorphism between $\dk((\hb_1))((\hb_2))|_{D^*_{\epsilon}}$ and
$\W_2^{\Z_2}((\hb_1))((\hb_2))|_{D^*_{\epsilon}}$ compatible with
restriction to $D^*_{\epsilon'}$ for any $\epsilon' < \epsilon$, the
isomorphism of algebras between
$(\text{C}^{\infty}(N^*_{\alpha,U})^{\Z_2}((\hb_1))((\hb_2)),\star)$
and
$(\text{C}^{\infty}(N^*_{\alpha,U})^{\Z_2}((\hb_1))((\hb_2)),\star_W)$
is compatible with restriction to $N^*_{\beta,V}$ for $V \subset U$
and $\beta \leq \alpha$. One then appeals to Theorem 5.6 of
~\cite{K} to note that since
$(\text{C}^{\infty}(N^*_{\alpha,U})^{\Z_2}((\hb_1))((\hb_2)),\star_W)$
and $\Gamma(N^*_{\alpha,U},\Aa_{M^-}((\hb_1))((\hb_2)))^{\Z_2}$ are
quantizations of $\text{C}^{\infty}(N^*_{\alpha,U})$ having the same
characteristic class, they are isomorphic. The desired proposition
therefore follows once we see that the isomorphism from Theorem 5.6
of ~\cite{K} is compatible with restriction to $N^*_{\beta,V}$. This
follows from the fact that Kravchenko's construction \cite{K} is a
local construction using a Fedosov type connection.
\end{proof}
Denote the $\mathcal B_{\epsilon}$-presheaf of algebras on $M$ coming from Proposition ~\ref{p5} by $\mathcal F$. One constructs a presheaf of algebras $\Aad((\hb_1))((\hb_2))$ on $M$ by setting
$$\Gamma(U,\Aad((\hb_1))((\hb_2))):= \varprojlim_{\stackrel{V \subset U}{V \in \mathcal B_{\epsilon}}} \Gamma(V,\mathcal F) \text{.}$$
Clearly,
\begin{equation} \la{a19} \Gamma(V, \Aad((\hb_1))((\hb_2))) \cong \Gamma(V,\mathcal F) \text{ for } V \in \mathcal B_{\epsilon}\text{.}\end{equation}
\begin{prop} \la{p6}
$$ \mathfrak{A}_{M/\Z_2}((\hb_1))((\hb_2)) \cong \Gamma(M,\Aad((\hb_1))((\hb_2)))^{\Z_2} \text{.}$$
\end{prop}
\begin{proof}
An element of $\Gamma(M,\Aad((\hb_1))((\hb_2)))^{\Z_2}$ is equivalent to a pair $(\alpha, \beta)$ such that \\$\alpha \in \Gamma(M^-,\Aad((\hb_1))((\hb_2)))^{\Z_2}$, $\beta \in \Gamma(N_{M_2^{\gamma},\epsilon},\Aad((\hb_1))((\hb_2)))$ and $\beta|_{N^*_{M_2^{\gamma},\epsilon}}$ coincides with $\alpha|_{N^*_{M_2^{\gamma},\epsilon}}$ under the isomorphism of Proposition 4.1 of ~\cite{HT}. By ~\eqref{a19}, this is precisely how elements of $\mathfrak{A}_{M/\Z_2}((\hb_1))((\hb_2))$ are constructed in ~\cite{HT}. That the isomorphism of vector spaces so obtained is an isomorphism of algebras is also immediate from the construction of $\mathfrak{A}_{M/\Z_2}((\hb_1))((\hb_2))$ from ~\cite{HT}.
\end{proof}

\subsubsection{Some important recollections and observations.} We recall that the algebra $\text{C}^{\infty}(M)$ comes equipped with a canonical bornology (see Appendix A6 of ~\cite{PPTT}). As a result, the algebra $\text{C}^{\infty}(M)^{\Z_2}$ is also equipped with the subspace bornology (which we will still call canonical). This induces bornologies on $\A_{M/\Z_2}((\hb_1))$ and $\mathfrak{A}_{M/\Z_2}((\hb_1))((\hb_2))$. Similarly, sections of the presheaf $\Aad((\hb_1))((\hb_2))$ (and other similarly constructed presheaves described in the following paragraphs) over each open $V \subset M$ are equipped with the canonical bornology. The isomorphism in Prop~\ref{p6} is an isomorphism of bornological $\c((\hb_1))((\hb_2))$-modules.\\

Further, note that for $U \subset M_2^{\gamma}$, $U$ open, the product $\star$ constructed in Section 3.1.1 on $\text{C}^{\infty}(N_{\alpha,U})^{\Z_2}((\hb_1))((\hb_2))$ restricts to a product on
$\text{C}^{\infty}(N_{\alpha,U})^{\Z_2}((\hb_1))[[\hb_2]]$. This is because the star product on $\text{C}^{\infty}(\R^2)^{\Z_2}((\hb_1))((\hb_2))$ coming from $\dk((\hb_1))((\hb_2))$ restricts to a product on $\text{C}^{\infty}(\R^2)^{\Z_2}((\hb_1))[[\hb_2]]$. One also notes that the restriction map from $\text{C}^{\infty}(N_{\alpha,U})^{\Z_2}((\hb_1))((\hb_2))$ to $\Gamma(N^*_{\alpha,U},\Aa_{M^-}((\hb_1))((\hb_2)))$ maps $\text{C}^{\infty}(N_{\alpha,U})^{\Z_2}((\hb_1))[[\hb_2]]$ to\\ $\Gamma(N^*_{\alpha,U},\Aa_{M^-}((\hb_1))[[\hb_2]])$. This is immediate once we note that the isomorphism between $\dk((\hb_1))((\hb_2))|_{D^*_{\epsilon}}$ and $\W_2^{\Z_2}((\hb_1))((\hb_2))|_{D^*_{\epsilon}}$ used to construct the above restriction map in Proposition 4.1 of ~\cite{HT} is of the form $\text{id}+\hb_2\nu_1+\text{ higher order terms in }\hb_2$.\\

One may therefore follow Section 3.1 (Steps 3.1.1-3.1.3) to
construct a sub-presheaf $\Aad((\hb_1))[[\hb_2]]$ of the presheaf of
algebras $\Aad((\hb_1))((\hb_2))$ constructed in Sections 3.1.1-3.1.3.
We denote $\Gamma(M,\Aad((\hb_1))[[\hb_2]])^{\Z_2}$ by $\mathfrak{A}_{M/\Z_2}((\hb_1))[[\hb_2]]$.
We may view $\mathfrak{A}_{M/\Z_2}((\hb_1))[[\hb_2]]$ as a filtered
algebra with $F_{-p}\mathfrak{A}_{M/\Z_2}:= \hb_2^p\mathfrak{A}_{M/\Z_2}((\hb_1))[[\hb_2]]$
for $p \geq 0$. The associated graded algebra with respect
to this filtration is clearly $\A_{M/\Z_2}((\hb_1))[[\hb_2]]$.
Similarly, the presheaf $\Aad((\hb_1))[[\hb_2]]$ may be viewed
as a filtered presheaf of algebras with $F_{-p}\Aad((\hb_1))[[\hb_2]]:=\hb_2^p\Aad((\hb_1))[[\hb_2]]$ for $p \geq 0$.
The associated graded presheaf of algebras with respect to this filtration is simply
the presheaf $\Aa((\hb_1))[[\hb_2]]$ constructed following Sections 3.1.1-3.1.3 with $\dk((\hb_1))((\hb_2))$ replaced by $\W_2^{\Z_2}((\hb_1))[[\hb_2]]$. One notes that $\Gamma(M,\Aa((\hb_1))[[\hb_2]])^{\Z_2} \cong \A_{M/\Z_2}((\hb_1))[[\hb_2]]$.\\

\subsection{Computation of Hochschild homology.}

For the rest of the paper, any Hochschild chain complex of an algebra is a Hochschild chain complex of a bornological algebra. The algebras we are working with are all equipped with canonical bornologies (which coincide with their precompact as well as with their von Neumann bornologies: see Appendix A of ~\cite{PPTT} for definitions). Recall that if $A$ is a bornological algebra, the (bornological) space of Hochschild $k$-chains of $A$ is the $k+1$-st bornological tensor power of $A$. The reader may refer to Appendix A of ~\cite{PPTT} for an excellent introduction to the homological algebra of bornological algebras and modules.
Clearly, the assignment
$$U \mapsto \text{C}_{\bullet}(\Gamma(U,\Aad((\hb_1))((\hb_2))))$$
gives a complex of presheaves of $\c((\hb_1))((\hb_2))$-modules on $M$. We denote this complex of presheaves by $\text{C}_{\bullet}(\Aad((\hb_1))((\hb_2)))$.
\begin{prop} \la{p7}
The natural map
$$\hh_{\bullet}(\mathfrak{A}_{M/\Z_2}((\hb_1))((\hb_2))) \rar {\mathbb H}^{-\bullet}(M,\text{C}_{\bullet}(\Aad((\hb_1))((\hb_2)))^{\Z_2}$$
is an isomorphism of $\c((\hb_1))((\hb_2))$-modules.
\end{prop}

\begin{proof}
It is clear that
$$\text{C}_{\bullet}(\mathfrak{A}_{M/\Z_2}((\hb_1))((\hb_2))) = \varinjlim_{p \geq 0} \hb_2^{-p}\text{C}_{\bullet}((\mathfrak{A}_{M/\Z_2}((\hb_1))[[\hb_2]],\star)) \text{.}$$
Further, as complexes of presheaves on $M$,
$$ \text{C}_{\bullet}(\Aad((\hb_1))((\hb_2))) \cong \varinjlim_{p \geq 0} \hb_2^{-p}\text{C}_{\bullet}((\Aad((\hb_1))[[\hb_2]],\star)) \text{.}$$
Further, the above isomorphism of complexes of presheaves is $\Z_2$-equivariant. Since cohomology commutes with direct limits, the desired proposition follows from Proposition ~\ref{p8}, which follows.
\end{proof}

\begin{prop} \la{p8}
The natural map
$$\hh_{\bullet}((\mathfrak{A}_{M/\Z_2}((\hb_1))[[\hb_2]],\star)) \rar {\mathbb H}^{-\bullet}(M,\text{C}_{\bullet}((\Aad((\hb_1))[[\hb_2]],\star)))^{\Z_2}$$
is an isomorphism of $\c((\hb_1))[[\hb_2]]$-modules.
\end{prop}
\begin{proof}
Pick a $\Z_2$-invariant finite cover $\mathfrak{V}$ of $M$ by sufficiently small open sets. Then,\\ ${\mathbb H}^{-\bullet}(M,\text{C}_{\bullet}((\Aad((\hb_1))[[\hb_2]],\star)))^{\Z_2}$ is the cohomology of the complex of $\Z_2$-invariant cochains in
$$ \text{Tot}^{\oplus}(\hat{C}_{\mathfrak{V}}(\text{C}_{\bullet}((\Aad((\hb_1))[[\hb_2]],\star)))) \text{.}$$
Note that $\text{C}_{\bullet}((\mathfrak{A}_{M/\Z_2}((\hb_1))[[\hb_2]],\star))$ is a filtered complex with $$F_{-p}\text{C}_{\bullet}((\mathfrak{A}_{M/\Z_2}((\hb_1))[[\hb_2]],\star)):= \hb_2^p\text{C}_{\bullet}((\mathfrak{A}_{M/\Z_2}((\hb_1))[[\hb_2]],\star))$$ for $p \geq 0$. Similarly, $\text{C}_{\bullet}((\Aad((\hb_1))[[\hb_2]],\star))$ is a filtered complex with
$$F_{-p}\text{C}_{\bullet}((\Aad((\hb_1))[[\hb_2]],\star)):=\hb_2^p\text{C}_{\bullet}((\Aad((\hb_1))[[\hb_2]],\star))$$ for $p \geq 0$.
The natural map
$$\text{C}_{\bullet}((\mathfrak{A}_{M/\Z_2}((\hb_1))[[\hb_2]],\star)) \rar \text{Tot}^{\oplus}(\hat{C}_{\mathfrak{V}}(\text{C}_{\bullet}((\Aad((\hb_1))[[\hb_2]],\star))))^{\Z_2}$$ is compatible with filtrations, the filtration on the right being induced by the filtration on $\text{C}_{\bullet}((\Aad((\hb_1))[[\hb_2]],\star))$. The $E^1_{pq}$-terms of the corresponding spectral sequences are
$\hb_2^{-p}\hh_{p+q}(\A_{M/\Z_2}((\hb_1)))$ and $\hb_2^{-p}\text{H}_{p+q}(\text{Tot}^{\oplus}(\hat{C}_{\mathfrak{V}}(\text{C}_{\bullet}(\Aa((\hb_1)))))^{\Z_2})$ respectively. The spectral sequences on both sides converge by the complete convergence theorem. It therefore, remains to verify that the induced map on $E^1_{pq}$ terms described above is an isomorphism. Let $\mathcal C_p$ denote the $\mathcal B_{\epsilon}$-presheaf of algebras such that $\mathcal C_p(V)=\text{C}^{\infty}(V)$ and $\mathcal C_p(N_{\alpha,U})=\text{C}^{\infty}(N_{\alpha,U})^{\Z_2}$. As in the step immediately after Proposition~\ref{p5}, we construct a presheaf of algebras out of $\mathcal C_p$, which we shall continue to denote by $\mathcal C$. An argument essentially identical to that used to reduce Proposition~\ref{p7} to the above verification can be further used to reduce the above verification to a verification that the natural map between $\hh_{\bullet}(\text{C}^{\infty}({M/\Z_2}))$ and
$\text{H}_{\bullet}(\text{Tot}^{\oplus}(\hat{C}_{\mathfrak{V}}(\text{C}_{\bullet}(\mathcal C)))^{\Z_2})$ is an isomorphism. Let $\pi:M \rar M/\Z_2$ denote the natural projection. Note that the cover $\mathfrak{V}:=\{V_{\alpha}\}$ of $M$ is such that the open subsets $\pi(V_{\alpha})$ form an open cover $\mathfrak{W}$ of $M/\Z_2$ (of course, if $\pi(V_{\alpha})=\pi(V_{\beta})$ for some $\alpha \neq \beta$, we count $\pi(V_{\alpha})$ only once in $\mathfrak{W}$). Note that the homology $\text{H}_{\bullet}(\text{Tot}^{\oplus}(\hat{C}_{\mathfrak{V}}(\text{C}_{\bullet}(\mathcal C)))^{\Z_2})$ coincides with $\text{H}_{\bullet}(\text{Tot}^{\oplus}(\hat{C}_{\mathfrak{W}}(\text{C}_{\bullet}(\mathcal C^{\infty}_{M/\Z_2}) )))$. The latter is precisely the hypercohomology ${\mathbb H}^{-\bullet}(M/\Z_2,\text{C}_{\bullet}(\mathcal C^{\infty}_{M/\Z_2}))$. Since $\text{C}_{\bullet}(\mathcal C^{\infty}_{M/\Z_2})$ is a complex of fine presheaves on $M/\Z_2$, the hypercohomology ${\mathbb H}^{-\bullet}(M/\Z_2,\text{C}_{\bullet}(\mathcal C^{\infty}_{M/\Z_2}))$ is precisely the homology $\text{H}_{\bullet}(\Gamma(M/\Z_2,\text{C}_{\bullet}(\mathcal C^{\infty}_{M/\Z_2})))$ of the complex of global sections of $\text{C}_{\bullet}(\mathcal C^{\infty}_{M/\Z_2})$. The Hochschild $k$-chains in this complex are germs along the principal diagonal $X \rar X^k$ of smooth functions on $X^k$, where $X=M/\Z_2$. We are therefore, reduced to showing that the Hochschild complexes $\text{C}_{\bullet}(\text{C}^{\infty}_{M/\Z_2})$ and $\Gamma(M/\Z_2,\text{C}_{\bullet}(\mathcal C^{\infty}_{M/\Z_2}))$ give the same homology. The localisation argument in Section 3 of ~\cite{Te}, which works for $M/\Z_2$ as well, does exactly this.
\end{proof}

We postpone the proof of the following proposition to Section 4. Let $i_l:M^{\gamma}_{2l} \rar M$ denote the natural inclusion from the collection $M^{\gamma}_{2l}$ of codimension $2l$-components of $M^{\gamma}$ into $M$

\begin{prop} \la{qism}
As complexes of presheaves on $M$, $\text{C}_{\bullet}(\Aad((\hb_1))((\hb_2))$ is quasiisomorphic to $\underline{\c}((\hb_1))((\hb_2))[2n] \bigoplus \oplus_{l \geq 1} i_{l*}\underline{\c}((\hb_1))((\hb_2))[2n-2l]$.
\end{prop}

The following theorem is an immediate consequence of Proposition~\ref{p7} and Proposition~\ref{qism}.

\begin{theorem} \la{three}
\begin{equation}\la{a20} \hh_{\bullet}(\mathfrak{A}_{M/\Z_2}((\hb_1))((\hb_2)))
\cong (\text{H}^{2n-\bullet}(M,\c)^{\Z_2} \bigoplus \oplus_{l \geq
1}
 \text{H}^{2n-2l-\bullet}(M_{2l}^{\gamma},\c))((\hb_1))((\hb_2)) \text{.}\end{equation}
\end{theorem}
We remark that the right hand side of ~\eqref{a20} is the Chen-Ruan (or "stringy") cohomology $\text{H}^{2n-\bullet}_{\text{CR}}(M/\Z_2,\c)((\hb_1))((\hb_2))$ of the orbifold $M/\Z_2$.

\subsection{Hochschild cohomology of $\mathfrak{A}_{M/\Z_2}$}.

Given Theorem~\ref{three}, one is initially tempted to follow the Van Den Bergh duality arguments in Proposition 6 of ~\cite{DE} and then appeal to Van Den Bergh duality to compute $\hh^{\bullet}(\mathfrak{A}_{M/\Z_2}((\hb_1))((\hb_2)),\mathfrak{A}_{M/\Z_2}((\hb_1))((\hb_2)))$. However, even in the bornological setting, it is not clear that $\mathfrak{A}_{M/\Z_2}((\hb_1))((\hb_2))$ has a bimodule resolution by finitely generated projective $\mathfrak{A}_{M/\Z_2}((\hb_1))((\hb_2))$-bimodules. We therefore, adopt different method inspired by ~\cite{Menc} to show that one has a Van Den Bergh duality type isomorphism relating Hochschild cohomology to homology in our situation.\\

Let $A:=\A_M((\hb))$ be the Fedosov quantization of $\text{C}^{\infty}(M)$ with Weyl curvature $\omega$. Consider the algebra $B:=A \rtimes \Z_2$. Following ~\cite{PPT} and ~\cite{EF}, one can construct trace density maps
$$\chi_l: \text{C}_{\bullet}(B) \rar \Omega^{2n-2l-\bullet}_{M^{\gamma}_{2l}} $$
for each $l \geq 0$. Here, for $l > 1$, $M^{\gamma}_{2l}$ denotes the collection of components of $M^{\gamma}$ of codimension $2l$. By convention,
$M^{\gamma}_0:=M$. Recall, for instance, from ~\cite{DE} that $\text{HH}_{\bullet}(B)= (\text{HH}_{\bullet}(A) \oplus \text{HH}_{\bullet}(A,A_{\gamma}))^{\Z_2}$. The map induced on homologies by $\chi_0$, after this identification, coincides with the map induced on homologies by the trace denisty $\chi_{FFS}$ that one may construct in the situation of ~\cite{FFS} following ~\cite{EF}. We also remark that the constructions of the $\chi_l$ are local in nature on $X:=M/\Z_2$. Let $\theta$ be a Hochschild $2n$-cycle of $B$ such that $\chi_0(\theta)=1$. The cap product with $\theta$ gives a map of complexes
$$\theta \cap -: \text{C}^{\bullet}(B,B) \rar \text{C}_{2n-\bullet}(B) \text{.}$$

\begin{prop} \la{vdb1} $\theta \cap -$ is a quasi-isomorphism. \end{prop}

\begin{proof}
 Recall that $X:=M/\Z_2$. Let $\C_{B,\bullet}$ denote the chain complex of sheaves on $X$ whose $k$-chains are given by the sheaf associated with the presheaf $$ U \mapsto \text{C}_{k}(\A(\pi^{-1}(U)) \rtimes \Z_2) \text{.}$$ Here, $\pi:M \rar X$ is the natural projection. Let $\C^{\bullet}_B$ denote the cochain complex of sheaves (with Hochschild coboundary) whose sheaf of $k$-cochains is the sheafification of the presheaf
$$ U \mapsto \text{C}^{k}_{loc}(B(U),B(U)):= \text{Hom}_{loc}((\A_{cpt}(\pi^{-1}(U)) \rtimes \Z_2)^{\hat{\otimes}k},\A(\pi^{-1}(U)) \rtimes \Z_2) \text{.}$$
Here the subscript ``loc" means that we only include those cochain $\Psi$ of $\text{Hom}((\A_{cpt}(\pi^{-1}(U))$ $ \rtimes \Z_2)^{\hat{\otimes} k},\A(\pi^{-1}(U)) \rtimes \Z_2)$ such that \begin{equation} \la{loccond} \pi(\text{supp } \Psi(a_1,..,a_k)) \subset \cap_{i=1}^k \pi(\text{supp } a_i)\text{.}\end{equation} Let $\text{C}^{\bullet}_{loc}(B,B)$ denote the subcomplex of $\text{C}^{\bullet}(B,B)$ comprising of cochains satisfying the locality condition~\eqref{loccond}. By Proposition 4.1 of ~\cite{PPTT}, the natural map of complexes $\text{C}^{\bullet}_{loc}(B,B) \rar \text{C}^{\bullet}(B,B)$ is a quasi-isomorphism. Using a spectral sequence argument as well as a localisation procedure similar to that in \cite{Te}, ~\cite{NPPT} show that the natural maps $\text{C}_{\bullet}(B) \rar \C_{B,\bullet}(X)$ and $\text{C}^{\bullet}_{loc}(B,B) \rar \C^{\bullet}_{B}(X)$ are quasi-isomorphisms. Further, note that cap product with $\theta|_{U}$ gives a map of complexes from $\C^{\bullet}_{B}(U)$ to $\C_{B,2n-\bullet}(U)$ for each open $U \subset X$. To see why this is the case, observe that even though $\theta|_{U}$ is not compactly supported, its cap product with an element $\Psi$ of $\text{C}^{k}_{loc}(B(U),B(U))$ is well defined. This is because $\Psi(a_1,..,a_k)$ makes sense even if $a_1 \otimes ...\otimes a_k$ is not compactly supported: the condition~\eqref{loccond} is equivalent to the condition that for any $x$ in $\pi^{-1}(U)$, $\psi(a_1,...,a_k)(x)$ depends only on the jets of $a_1,..,a_k$ at $x$ and at $\gamma.x$ (see, for instance, the proof of Proposition 4.11 of ~\cite{PPTT}). It follows that $\theta \cap -$ is defined as a map of complexes of sheaves from $\C^{\bullet}_B$ to $\C_{B,2n-\bullet}$. Hence, one has the following commutative diagram,
$$\begin{CD}
\text{C}^{\bullet}(B,B) @>\theta \cap ->> \text{C}_{2n-\bullet}(B)\\
  @AAA       @AA{\text{id}}A\\
\text{C}^{\bullet}_{loc}(B,B) @>\theta \cap ->> \text{C}_{2n-\bullet}(B)\\
@VVV   @VVV\\
\C^{\bullet}_B(X) @>>\theta \cap -> \C_{B,2n-\bullet}(X) \\
\end{CD}.$$

Since the vertical arrows in the above diagram are quasi-isomorphisms, it suffices to check that the bottom arrow is a quasi-isomorphism. Since $\C_{B,\bullet}$ as well as $\C^{\bullet}_{B}$ are complexes of fine sheaves on $X$, the bottom arrow in the above diagram is the map induced by $\theta \cap -$ on hypercohomologies. It therefore, suffices to verify that $\theta \cap -:\C^{\bullet}_B \rar \C_{B,2n-\bullet}$ is a quasi-isomorphism of complexes of sheaves on $X$. This is a local verification. We outline it for a sufficiently small Darboux neighborhood $U:=V/\Z_2$ of a point on the singular locus of $X$. The (more straightforward) verification for $U$ a neighborhood of a point not contained in the singular locus is left to the reader.\\

Let $V$ be have Darboux coordinates $p_1,q_1,..,p_n,q_n$ with $x_{2i-1}:=p_i$ and $x_{2i}:=q_i$. By Theorem 5.5.1 of ~\cite{F}, we may assume that $\A_V$ is isomorphic to a quantum algebra with trivial Fedosov connection. Note that $\theta|_{U}$ is a Hochschild $2n$-cycle of $\A_V \rtimes \Z_2$ such that $\chi_0(\theta|_{U})=1$. We need to check that $\theta|_{U} \cap -$ is a quasiisomorphism. Recall that $\text{HH}_{\bullet}(\A_V \rtimes \Z_2)=(\text{HH}_{\bullet}(\A_V) \oplus \text{HH}_{\bullet}(\A_V,\A_{V,\gamma}))^{\Z_2}$. Consider the $\Z_2$-invariant (reduced) cycle
$$c_{2n}:= 1 \otimes \sum_{\sigma \in S_{2n}} \text{sgn}(\sigma) x_{\sigma(1)} \otimes....\otimes x_{\sigma(2n)} \text{.}$$
Further recall that
$$\text{HH}_{loc}^{\bullet}(\A_{cpt,V} \rtimes \Z_2,\A_V \rtimes \Z_2)= (\text{HH}^{\bullet}_{loc}(\A_{cpt,V},\A_V) \oplus \text{HH}^{\bullet}_{loc}(\A_{cpt,V},\A_{V,\gamma}))^{\Z_2}.
$$
Since  $\chi_0(\theta|_{U})=1$ and  $\chi_{FFS}(c_{2n})=1$, our verification will be complete once we check that $c_{2n} \cap -: \text{C}^{\bullet}_{loc}(\A_{cpt,V},\A_V) \rar \text{C}_{2n-\bullet}(\A_V)$ and $c_{2n} \cap -: \text{C}^{\bullet}_{loc}(\A_{cpt,V},\A_{V,\gamma}) \rar
\text{C}_{2n-\bullet}(\A_V,\A_{V,\gamma})$ are quasi-isomorphisms. $c_{2n} \cap -: \text{C}^{\bullet}_{loc}(\A_{cpt,V},\A_V) \rar \text{C}_{2n-\bullet}(\A_V)$ is indeed a quasi-isomorphism: the only nontrivial cohomology on the left hand side is generated by the $0$-cocycle $1$, and $c_{2n} \cap 1=c_{2n}$.\\

It remains to check that $c_{2n} \cap -: \text{C}^{\bullet}_{loc}(\A_{cpt,V},\A_{V,\gamma}) \rar
\text{C}_{2n-\bullet}(\A_V,\A_{V,\gamma})$ is a quasi-isomorphism. For this, we assume without loss of generality that $\gamma$ fixes $x_1,...,x_{2n-2l}$. As in the proof of Proposition 4.4 of ~\cite{PPTT}, one still has a Koszul resolution of $\A_{cpt,V}$ as $\A_{cpt,V} \otimes \A_{cpt,V}^{op}$-bimodules. Using this resolution, $\text{HH}^{\bullet}(\A_{cpt, V},\A_{V,\gamma})$ may be calculated as the cohomology of the complex
$$K^p_{\gamma}:= \wedge^p W \otimes \A_V $$
with
$$d_{\gamma}(a \otimes \partial y_{i_1} \wedge ....\wedge \partial y_{i_p}):= \sum_{j=1}^{2n} {(-1)}^j (y_j \star a-a \star_{\gamma} y_j) \partial y_j \wedge \partial y_{i_1} \wedge ....\wedge \partial y_{i_p} \text{.}$$
Here, $y_i:=x_{2i+1}$ for $1 \leq i \leq n$ and $y_i:=x_{2(i-n)}$ for $i=n+1,..,2n$. $W$ denotes the linear span of the $\partial y_i$'s. Further, one has a projection from the (reduced) cochain complex $\text{C}^{\bullet}_{loc}(\A_V,\A_{V,\gamma})$ to $K^{\bullet}_{\gamma}$. Moreover, there exists a generator $\Psi_{\gamma}$ of the unique nonzero cohomology of $\text{C}^{\bullet}_{loc}(\A_V,\A_{V,\gamma})$ whose restriction to $\wedge^{2l}W^*$ coincides with $\partial y_{2n-2l+1} \wedge .... \wedge \partial y_{2n}$. Now, $c_{2n}$ may be viewed as the shuffle product of $c_{2n-2l}$ formed out of $x_1,..,x_{2n-2l}$ and $c_{2l}$ formed out of $x_{2n-2l+1},....,x_{2n}$. This makes it easy to see that $$ c_{2n} \cap \Psi_{\gamma} = c_{2n-2l} $$
where the right hand side of the above equation is formed our of $x_1,..,x_{2n-2l}$. The latter is the generator of the (unique) nonzero homology $\text{HH}_{2n-2l}(\A_V,\A_{V,\gamma})$.
\end{proof}

\begin{theorem} \la{four}
$$\hh^{\bullet}(\mathfrak{A}_{M/\Z_2}((\hb_1))((\hb_2)),\mathfrak{A}_{M/\Z_2}((\hb_1))((\hb_2))) \cong \text{H}^{\bullet}_{\text{CR}}(M/\Z_2,\c)((\hb_1))((\hb_2)) \text{.}$$
\end{theorem}

\begin{proof}
 Let $A_0:=\A_{M/\Z_2}((\hb_1))$. Let $P:=\textbf{e}B$ and let $Q:=B\textbf{e}$ where $\textbf{e}:=\frac{1+\gamma}{2}$. There are natural bimodule isomorphisms $u: P \otimes_B Q \rar A_0$ and $v:Q \otimes_{A_0} P \rar B$. One checks that $(A_0,B,P,Q,u,v)$ forms a Morita context in the sense of Appendix A7 of ~\cite{PPTT}. It follows that $A_0$ and $B$ are Morita equivalent as bornological algebras.\\

 Since $\text{HH}_{2n}(B) \cong \c((\hb_1))$, Proposition~\ref{vdb1} can be restated to say that the map
 \begin{eqnarray}
  \nonumber  \text{HH}_{2n}(B) \otimes_{\c((\hb_1))} \text{HH}^{\bullet}(B,B) \rar \text{HH}_{2n-\bullet}(B)\\
  \la{vdb2}    \alpha \otimes \beta \mapsto \alpha \cap \beta
 \end{eqnarray}
 of $\c((\hb_1))$-modules is an isomorphism. Since $A_0$ and $B$ are Morita equivalent as bornological algebras,
 the map
 \begin{eqnarray}
  \nonumber  \text{HH}_{2n}(A_0) \otimes_{\c((\hb_1))} \text{HH}^{\bullet}(A_0,A_0) \rar \text{HH}_{2n-\bullet}(A_0)\\
  \la{vdb3}    \alpha \otimes \beta \mapsto \alpha \cap \beta
 \end{eqnarray}
 is an isomorphism of $\c((\hb_1))$-modules.\\

 Note that $\text{C}_{\bullet}(\mathfrak{A}_{M/\Z_2}((\hb_1))((\hb_2)))$ as well as $\text{C}^{\bullet}(\mathfrak{A}_{M/\Z_2}((\hb_1))((\hb_2)),\mathfrak{A}_{M/\Z_2}((\hb_1))((\hb_2)))$ are complexes filtered by powers of $\hb_2$. By Theorem~\ref{three} and $\c((\hb_1))((\hb_2))$-linearity, there exists a $2n$-cycle $\theta$ in $\text{C}_{\bullet}(\mathfrak{A}_{M/\Z_2}((\hb_1))[[\hb_2]])$ representing a nonzero homology class in $\text{HH}_{2n}(\mathfrak{A}_{M/\Z_2}((\hb_1))((\hb_2)))$. We claim that for a suitable choice of $\theta$, $\theta \cap -$ is an isomorphism.
 Let $\theta_{0}$ denote the image of $\theta$ under the map of chain complexes induced by the homomorphism $\mathfrak{A}_{M/\Z_2}((\hb_1))[[\hb_2]]$ mapping $\hb_2$ to $0$. Note that $\theta \cap -$ is a map of filtered complexes whose associated graded map is $\theta_0 \cap -$. Since the map~\eqref{vdb3} is an isomorphism, $\theta_0 \cap -$ is an isomorphism whenever $[\theta_0] \neq 0$.
 We therefore, need to show the existence of a $2n$-cycle $\theta$ of $\text{C}_{\bullet}(\mathfrak{A}_{M/\Z_2}((\hb_1))[[\hb_2]])$ such that $[\theta]$ is nonzero in $\text{HH}_{2n}(\mathfrak{A}_{M/\Z_2}((\hb_1))((\hb_2)))$ and $[\theta_0] \neq 0$. Proposition~\ref{vdbf} in Section 4.5 does exactly this, proving the desired theorem.

\end{proof}
\iffalse \noindent{\textbf{Remark.}} By Theorem~\ref{three}, any
generator $c$ of
$\text{HH}_{2n}(\mathfrak{A}_{M/\Z_2}((\hb_1))((\hb_2)))$ satisfies
$B[c]=0$, where $B$ is Connes' differential. Theorem 3.4.3 of
~\cite{G} tells us that the Hochschild cohomology
$\text{HH}^{\bullet}(\mathfrak{A}_{M/\Z_2}((\hb_1))((\hb_2)),\mathfrak{A}_{M/\Z_2}((\hb_1))((\hb_2)))$
equipped with the degree $-1$ operator $-B$ becomes a
Batalin-Vilkovisky algebra. We also remark that
Proposition~\ref{vdb1} may be generalized to algebras of quantum
functions of arbitrary symplectic orbifolds. Here too, Theorem 3.4.3
of ~\cite{G} tells us that $\text{HH}^{\bullet}(B,B)$, and
therefore, the Chen-Ruan cohomology of the orbifold, are
Batalin-Vilkovisky algebras. We have however, not checked the
validity of Theorem 3.4.3 of ~\cite{G} in detail for the topological
context in which we are working. \fi
\section{Traces.}
As in Section 2.4, $\K:=\c((\hb_1))((\hb_2))$.
We construct a trace density
$$\chi^{\gamma}:\text{C}_{\bullet}(\mathfrak{A}_{M/\Z_2}((\hb_1))((\hb_2))) \rar \Omega^{2n-2-\bullet}(M_2^{\gamma})((\hb_1))((\hb_2))\text{.}$$
This is a map of complexes of $\K$-vector spaces (with the right hand side equipped with the de-Rham differential). We prove an algebraic index theorem computing $\chi^{\gamma}(1)$. In fact, our construction of the trace density $\chi^{\gamma}$ extends to yield a map of complexes of presheaves $\text{C}_{\bullet}(\Aad((\hb_1))((\hb_2))) \rar i_{2*} \Omega^{2n-2-\bullet}_{M_2^{\gamma}}((\hb_1))((\hb_2))$ on $M$. This may be viewed as a map $\lambda:\text{C}_{\bullet}(\Aad((\hb_1))((\hb_2))) \rar i_{2*}\underline{\K}_{M_2^{\gamma}}[2n-2]$ in the derived category of presheaves on $M$. Similar trace densities $\lambda_l: \text{C}_{\bullet}(\Aad((\hb_1))((\hb_2))) \rar i_{2l*} \underline{\K}_{M_{2l}^{\gamma}}[2n-2l]((\hb_1))((\hb_2))$ can be constructed for $l >1$. The key parts of the latter construction are in ~\cite{PPT}. We skip the latter construction in order to avoid being repetitive. We also construct a map $\mu:\text{C}_{\bullet}(\Aad((\hb_1))((\hb_2))) \rar \underline{\K}_{M}[2n]$ in the derived category of presheaves on $M$. At the end of this section, we will show that $\lambda \oplus \mu$ is a quasi-isomorphism when $M^{\gamma}=M_2^{\gamma}$. More generally, the same method (with some more notation chasing) as that used in the end of this section shows that
$\oplus_{l>1} \lambda_l \oplus \lambda \oplus \mu$ is a quasi-isomorphism. This completes the proof of Theorem ~\ref{three}. The map $\mu$ mentioned above yields a trace $\chi_0$ on $\mathfrak{A}_{M/\Z_2}((\hb_1))((\hb_2))$ such that $\chi_0$ and $\chi^{\gamma}$ are linearly independent over $\c((\hb_1))((\hb_2))$. In Section 4.5, we discuss some basic properties of $\chi_0$. In the process, we obtain a proposition that completes the proof of Theorem~\ref{four}.

\subsection{Preliminaries.}
Notations in this section are as in Section 3.1.1. Recall that a sheaf $\Aa_{D}$ of algebras on $M_2^{\gamma}$ is constructed by constructing a Fedosov connection $D$ on $\bigwedge^{\bullet}T^*M_2^{\gamma} \otimes \Ww \otimes_{\c((\hb_1))} \V$. In ~\cite{HT}, $D$ was constructed starting with a symplectic connection $\nabla_{T}$ on $TM_2^{\gamma}$ and a Hermitian connection $\nabla_N$ on $N$. These induce connections $\partial_{T}$ and $\partial_{N}$ on $\Ww$ and $\V$ respectively. This makes $\nabla:=\partial_{T} \otimes 1 +1 \otimes \partial_N$ a connection on $\Ww \otimes_{\c((\hb_1))} \V$. $\nabla$ automatically extends to a connection on $\bigwedge^{\bullet}T^*M_2^{\gamma} \otimes \Ww \otimes_{\c((\hb_1))} \V$. The connection $D$ is then constructed as a sum
\begin{equation} \la{3.1} D=\nabla +[A,-]_{\hb_1} \end{equation}
where $A \in \Omega^1(M_2^{\gamma}, \Ww \otimes_{\c((\hb_1))} \V)$. Recall that $D^2=[\Theta,-]_{\hb_1}$ where $\Theta \in \Omega^2(M_2^{\gamma},\Ww \otimes_{\c((\hb_1))} \V)$ is central. This implies that $\Theta \in \Omega^2(M_2^{\gamma},\K)$.\\

Recall from equation ~\eqref{a11} that one has a map of Lie algebras $\mathfrak{sp}_{2n-2}(\K) \oplus \mathfrak{u}_{1}(\K) \rar \W_{n-1}((\hb_1)) \otimes_{\c((\hb_1))} \dk((\hb_1))((\hb_2))$. This induces a map $\mathfrak{sp}(TM_2^{\gamma}) \oplus \mathfrak{u}(N) \rar \Ww \otimes_{\c((\hb_1))} \V$ of bundles of Lie algebras over $M_2^{\gamma}$. We recall from ~\cite{HT} that
\begin{equation} \la{3.2} \nabla^2=[R_T+R_N,-]_{\hb_1} \end{equation}
where $R_T \in \Omega^2(M_2^{\gamma},\mathfrak{sp}(TM_2^{\gamma}))$ and $R_N \in \Omega^2(M_2^{\gamma},\mathfrak{u}(N))$. It follows that
\begin{equation} \la{3.3} \nabla A+\frac{1}{2}[A,A]_{\hb_1}=\Theta-R_T-R_N \text{.} \end{equation}
We remark that changing $\nabla_T$ and $\nabla_N$ changes $\nabla$ (and hence, $A$) by an element of $\Omega^1(M_2^{\gamma},\mathfrak{sp}(TM_2^{\gamma}) \oplus \mathfrak{u}(N))$.

\subsection{The first trace: construction of $\chi^{\gamma}$.}
The construction of $\chi^{\gamma}$ follows the trace density construction from ~\cite{EF}. Since this construction has been done in detail even in subsequent works (see ~\cite{PPT2},~\cite{W} and ~\cite{R},~\cite{S} for analogs of this construction for cyclic and Lie chains respectively), we just specify the steps in this construction without providing detailed proofs of their validity.\\

Let $U$ be an open subset of $M_2^{\gamma}$ on which $TM_2^{\gamma}$ and $N$ are trivial. Choosing trivializations of $TM_2^{\gamma}$ and $N$ over $U$, one obtains an identification of sheaves of DG-algebras

\begin{equation} \la{3.4} (\bigwedge^{\bullet}T^*M_2^{\gamma} \otimes \Ww \otimes_{\c((\hb_1))}\V,D)  \cong (\Omega^{\bullet}(U,\W_{n-1}((\hb_1)) \otimes_{\c((\hb_1))} \dk((\hb_1))((\hb_2))),d+[\theta,-]_{\hb_1}) \end{equation}

where $\theta \in \Omega^1(U,\W_{n-1}((\hb_1)) \otimes_{\c((\hb_1))} \dk((\hb_1))((\hb_2)))$ satisfies the Maurer-Cartan condition
$$d\theta +\frac{1}{2}[\theta,\theta]_{\hb_1} \in \Omega^2(U,\K) \text{.}$$
Recall that $\Aa_{D}$ is the sheaf of degree $0$ flat sections of $(\bigwedge^{\bullet}T^*M_2^{\gamma} \otimes \Ww \otimes_{\c((\hb_1))}\V,D)$. Let $\times$ denote the shuffle product on Hochschild chains and let $(\theta)^k$ denote the Hochschild chain $1\otimes \theta \otimes ...\otimes \theta$ with $k$ factors $\theta$. Following ~\cite{EF}, one constructs a map
\begin{eqnarray}
\nonumber \chi^{\gamma}: \text{C}_{\bullet}(\Gamma(U,\Aa_{D})) \rar \Omega^{2n-2-\bullet}_U((\hb_2,\hb_2))\\
\la{3.5} a \mapsto \sum_{k=0}^{\infty} {(-1)}^{\lfloor \frac{k}{2} \rfloor}\psi_{2n-2}(a \times (\theta)^k)
\end{eqnarray}
of complexes of sheaves on $U$. We remark that in ~\cite{EF}, $d\theta+\frac{1}{2}[\theta,\theta]=0$. However, as pointed out in ~\cite{PPT2}, equation ~\eqref{3.5} gives a map of complexes in our situation as well. Moreover, a different choice of trivialization of $TM_2^{\gamma}$ and of trivialization of $N$ changes $\theta$ by an element of $\Omega^1(U,\mathfrak{sp}_{2n-2}(\K) \oplus \mathfrak{u}_{1}(\K))$. Since the cocycle $\psi_{2n-2}$ is $\mathfrak{sp}_{2n-2}(\K) \oplus \mathfrak{u}_{1}(\K)$-basic (see Proposition ~\ref{p4}), this change leaves the map $\chi^{\gamma}$ unchanged. As a result, the map $\chi^{\gamma}$ is well defined globally on $M_2^{\gamma}$, giving us a map of complexes
$$\chi^{\gamma}:\text{C}_{\bullet}(\mathfrak{A}_{M/\Z_2}((\hb_1))((\hb_2))) \rar \Omega^{2n-2-\bullet}(M_2^{\gamma})((\hb_1))((\hb_2))$$
(note that there is a natural map of algebras $\mathfrak{A}_{M/\Z_2}((\hb_1))((\hb_2)) \rar \Aa_{D}$). Also, after a choice of trivialization of  $TM_2^{\gamma}$ and of $N$ over $U$, $A$ differs from $\theta$ by an element of $\Omega^1(U,\mathfrak{sp}_{2n-2}(\K) \oplus \mathfrak{u}_{1}(\K))$. It follows that equation ~\eqref{3.5} may be rewritten as
\begin{equation} \la{3.6} a \mapsto \sum_{k=0}^{\infty}{(-1)}^{\lfloor \frac{k}{2} \rfloor}\psi_{2n-2}(a \times (A)^k)\text{.} \end{equation}

\subsubsection{A basic example.} Let us consider the case when $M \subset \R^{2n}$ is an open neighborhood of the origin with symplectic form
$$ \omega = \sum_{i=1}^{n-1} dp_i \wedge dq_i +\frac{1}{2i} dz \wedge d\bar{z} \text{.}$$
Let $\Z_2$ act on $M$ such that $(p,q,z,\bar{z}) \mapsto
(p,q,-z,-\bar{z})$. Then $M_2^{\gamma}$ is given by the equation
$z=\bar{z}=0$ with $p_1,..,p_{n-1},q_1,..,q_{n-1}$ being Darboux
coordinates for $M_2^{\gamma}$. In this case, the connection
$\triangledown:= \partial_T \otimes 1+1 \otimes \partial_N$ (see
Section 4.1) may be chosen to be $d$ itself. Then, $D:=d+\delta$ is
itself a (trivial) Fedosov connection where
$\delta=\sum_{i=1}^{2n-2}dx^i\frac{\partial}{\partial y^i}$. We
remark that here, $x^{2i-1}=p_i$ and $x^{2i}=q_i$ with $y^i$ being
the corresponding fiberwise coordinates. For this trivial Fedosov
connection, the form $\theta$ is given by $\sum_{1 \leq i,j \leq
2n-2} \omega_{ij}y^idx^j$ where $\omega_{ij}=1$ if
$\{i,j\}=\{2k-1,2k\}$ for some $k$ and $\omega_{ij}=0$ otherwise.
Therefore, $$(\theta)^{2n-2}=\sum_{\sigma \in S_{2n-2}}
\text{sgn}(\sigma) 1 \otimes y^{\sigma(1)} \otimes.... \otimes
y^{\sigma(2n-2)} dp_1 \wedge dq_1 \wedge .... \wedge dp_{n-1} \wedge
dq_{n-1} \text{.} $$ For any function on $M$ of the form
$F=f(p,q).g(z ,\bar{z})$, we see that
$$ \chi^{\gamma}(F)= \psi_{2n-2}(F \times (\theta)^{2n-2}) = \phi(g).\tau_{2n-2}(f \times (\theta)^{2n-2}) $$
$$ = \phi(g)f dp_1 \wedge dq_1 \wedge ....\wedge dp_{n-1} \wedge dq_{n-1} \text{.}$$
Here, $\phi(g)$ is given by the formula in Theorem ~\ref{two}. We may therefore, rewrite the above local formula as
\begin{equation} \la{localtr1}
\begin{split}
\chi^{\gamma}(F)= \sum_{k=0}^{\infty}
\frac{(i\hb_1)^k}{(k!)^2}.\large({\prod}_{l=1}^k{
\{\frac{l}{2}+(-1)^{l+1}\frac{2\lfloor\frac{l+1}{2}\rfloor\hb_2}{l+1}\large\}}\large).&
\frac{\partial^{2k} F}{\partial z^k \partial
\bar{z}^k}|_{z=\bar{z}=0} \\
& dp_1 \wedge dq_1 \wedge ....\wedge dp_{n-1} \wedge dq_{n-1}.
\end{split}
\end{equation}
One can then argue as in the end of the proof of Theorem ~\ref{two}
to show that equation ~\eqref{localtr1} is valid for arbitrary
functions on $M$ as well. Another point to note is that
\begin{equation} \la{localf2} \chi^{\gamma}(\sum_{\sigma \in S_{2n-2}} \text{sgn}(\sigma) 1 \otimes y^{\sigma(1)} \otimes.... \otimes y^{\sigma(2n-2)})=1 \text{.} \end{equation}

\subsubsection{A remark for general $M$.}
We return to the general case, as in the beginning of this section.
We recall from ~\cite{F} (see Section 5) that for every $x \in
M_2^{\gamma}$, the symplectic form $\omega$ on $M$ may be identified
with $\sum_{i=1}^{n-1} dp_i \wedge dq_i +\frac{1}{2i} dz \wedge
d\bar{z}$ over $V:=N_{\alpha,U}$ for sufficiently small $\alpha$ for
some neighborhood $U$ of $x$ in $M_2^{\gamma}$. We denote the
trivial Fedosov connection in Example 4.2.1 by $D^0$. We also denote
the corresponding sheaf of algebras on $U$ by $\Aa_{D^0}$. Further,
by imitating the proof of Theorem 5.5.1 in ~\cite{F1}, one can prove
the following analog of the ``quantum Darboux theorem" (Theorem
5.5.1 of ~\cite{F1}).
\begin{prop} \la{locdarb}
$\Gamma(W,\Aa_D)$ is locally isomorphic to $\Gamma(W,\Aa_{D^0})$ for
some neighborhood $W$ of $x$ contained in $U$.
\end{prop}

\subsection{The algebraic index theorem.}
Let $\Theta,R_T,R_N$ be as in Section 4.1.

\begin{theorem} \la{six}
The $2n-2$-form
$$\chi^{\gamma}(1)- \hb_1^{n-1}(\hat{A}({R_T})\text{Ch}(-\frac{\Theta}{\hb_1})\text{Ch}_{\phi}(\frac{R_N}{\hb_1}))_{n-1} $$
 is an exact form on $M_2^{\gamma}$.
 \end{theorem}
 \begin{proof}
 Let $P=(\hat{A}_{\hb_1}\text{Ch}\text{Ch}_{\phi})_{n-1}$. Then,
\[
\begin{split}
 \chi^{\gamma}(1)={(-1)}^{n-1}\psi_{2n-2}((A)^{2n-2})
 &=\frac{{(-1)}^{n-1}}{(2n-2)!}(\text{ev}_1\Psi_{2n-2})(A^{2n-2})\\
 &={(-1)}^{n-1}(\text{ev}_1\Psi_{2n-2})(A \wedge ...\wedge A) \text{.}
\end{split}
\]
  By Theorem ~\ref{five}, it follows that modulo exact forms,
 \begin{eqnarray*}
  \chi^{\gamma}(1)= \chi(P)(A \wedge .... \wedge A)\text{.}
  \end{eqnarray*}
  Hence, for any vector fields $\xi_1,...,\xi_{2n-2}$ on $M_2^{\gamma}$,
  \begin{eqnarray*}
  \chi^{\gamma}(1)(\xi_1,...,\xi_{2n-2})=\chi(P)(A(\xi_1) \wedge ... \wedge A(\xi_{2n-2}))\\
 =\frac{1}{(n-1)!}\sum_{\sigma(2i-1)<\sigma(2i)} {(-1)}^{\sigma}P(C(A(\xi_{\sigma(1)}),A(\xi_{\sigma(2)})),...,C(A_{\sigma(2n-3)},A_{\sigma(2n-2)}))\text{.}
 \end{eqnarray*}
 As in ~\cite{FFS}, $A$ may be chosen so that $\text{pr}(A)=0$. In this case, for smooth vector fields $\xi,\eta$ on $M$,
 $$C(A(\xi),A(\eta))=-\text{pr}[A(\xi),A(\eta)]_{\hb_1}= -\text{pr}(\nabla A(\xi,\eta)+[A(\xi),A(\eta)]_{\hb_1})= (\Theta-R_T-R_N)(\xi,\eta)\text{.}$$
 The last equality above is due to ~\eqref{3.3}.
 It follows that
 \begin{eqnarray*}
 \chi^{\gamma}(1)(\xi_1,...,\xi_{2n-2})&&= \frac{1}{(n-1)!}\sum_{\sigma(2i-1)<\sigma(2i)}
 {(-1)}^{\sigma}P(C(A(\xi_{\sigma(1)}),A(\xi_{\sigma(2)})),...\\
 &&\hspace{5cm}...,C(A_{\sigma(2n-3)},A_{\sigma(2n-2)}))\\
 &&=\frac{1}{(n-1)!}\sum_{\sigma(2i-1)<\sigma(2i)} {(-1)}^{\sigma} P((\Theta-R_T-R_N)(\xi_{\sigma(1)},\xi_{\sigma(2)}),...\\
 &&\hspace{5cm}...,(\Theta-R_T-R_N)(\xi_{\sigma(2n-3)},\xi_{\sigma(2n-2)}))\\
 &&=\frac{1}{(n-1)!} P((\Theta-R_T-R_N)^{n-1})(\xi_1,...,\xi_{2n-2})
\text{.}
 \end{eqnarray*}
 Recall that for any $X \in  \mathfrak{sp}_{2n-2}(\K) \oplus \mathfrak{gl}_{N}(\K) \oplus \mathfrak{u}_1(\K)$, $$\frac{1}{(n-1)!}P(X,..,X)=P(X)$$ where the element $P$ of $(S^{n-1}\mathfrak{h}^*)^{\mathfrak{h}}$ is being viewed as a linear form $\mathfrak{h}^{\otimes n-1}$ in the left hand side and a polynomial function on $\mathfrak{h}$ in the right hand side of the above equation. Hence, modulo exact forms,
 $$\chi^{\gamma}(1)= (\hat{A}_{\hb_1}({R_T})\text{Ch}(-{\Theta})\text{Ch}_{\phi}(R_N))_{n-1}=\hb_1^{n-1}(\hat{A}(R_T)\text{Ch}(-\frac{\Theta}{\hb_1})\text{Ch}_{\phi}(\frac{R_N}{\hb_1}))_{n-1}\text{.} $$
 This completes the proof of the desired result.
 \end{proof}

\subsection{Proof of Proposition~\ref{qism}.}

We will prove the proposition for the case when
$M^{\gamma}=M_2^{\gamma}$. The general case is proven in a
completely similar fashion (with some extra notation to keep track
of). For this case, we shall denote $i_2$ by $i$. Let
$j:M^-=M\setminus M^\gamma \rar M$ denote the natural inclusion. Let
$\text{C}_{\bullet}^{sh}(\Aad((\hb_1))((\hb_2)))$ denote the
sheafification of the complex of presheaves
$\text{C}_{\bullet}(\Aad((\hb_1))((\hb_2)))$. We note that
$j^*\text{C}_{\bullet}^{sh}(\Aad((\hb_1))((\hb_2)))=\text{C}_{\bullet}^{sh}(\Aa_{M^-}((\hb_1))((\hb_2)))$
where $\Aa_{M^-}((\hb_1))((\hb_2))$ is the Fedosov quantization of
the sheaf of smooth functions on $M^-$ with Weyl curvature $\omega$.
One can apply the trace density construction of ~\cite{EF} to extend
the construction in ~\cite{FFS} to a map of complexes
$\chi_{FFS}:\text{C}_{\bullet}^{sh}(\Aa_{M^-}((\hb_1))((\hb_2)))
\rar \Omega^{2n-\bullet}_{M^-}((\hb_1))((\hb_2))$. One has the
following composite map in the derived category
$\text{D}(\text{Sh}_{\K}(M))$ of sheaves of $\K$-vector spaces on
$M$, which we shall denote by $\mu^{sh}$.
 \begin{eqnarray*}
 \text{C}_{\bullet}^{sh}(\Aad((\hb_1))((\hb_2))) \rar Rj_*\text{C}_{\bullet}^{sh}(\Aa_{M^-}((\hb_1))((\hb_2))) \stackrel{Rj_*\chi_{FFS}}{\rar} Rj_*\Omega_{M^-}^{2n-\bullet}((\hb_1))((\hb_2))\\
 Rj_*\Omega_{M^-}^{2n-\bullet}((\hb_1))((\hb_2)) \cong Rj_*\underline{\K}_{M^-}[2n] \cong \underline{\K}[2n]
 \end{eqnarray*}
The last isomorphism in $\text{D}(\text{Sh}_{\K}(M))$ is because
$\underline{\K}$ is an injective sheaf on $M^-$ (which implies that
$Rj_*\underline{\K}_{M^-} \cong R^0j_*\underline{\K}_{M^-}$) and
also because $M_2^{\gamma}$ is of codimension $2$ (which implies
that $R^0j_*\underline{\K}_{M^-} \cong \underline{\K}$). As a
result, in the derived category $\text{D}(\text{PrSh}_{\K}(M))$ of
presheaves of $\K$-vector spaces on $M$, one obtains the following
composite map, which shall be denoted by $\mu$,
$$\text{C}_{\bullet}(\Aad((\hb_1))((\hb_2))) \rar  \text{C}_{\bullet}^{sh}(\Aad((\hb_1))((\hb_2))) \rar \underline{\K}[2n]. $$
Here, the first arrow is sheafification and the second arrow is forgetful functor applied to $\mu^{sh}$. Proposition~\ref{qism} is implied by the following stronger proposition.

\begin{prop} \la{qism2} $\mu \oplus \lambda: \text{C}_{\bullet}(\Aad((\hb_1))((\hb_2))) \rar \underline{\K}[2n] \oplus i_*\underline{\K}_{M_2^{\gamma}}[2n-2]$ is an isomorphism in $\text{D}(\text{PrSh}_{\K}(M))$.
\end{prop}

\begin{proof}
It suffices to verify that for sufficiently small open subsets $V$ of $M$, $(\mu \oplus \lambda)|_{V}$ induces an isomorphism on homologies. We will do this case by case.\\

{\it Case 1: $V \subset M^-$.} In this case, $i_*\underline{\K}_{M_2^{\gamma}}|_{V}=0$. We therefore, need to verify that $\mu|_{V}:\text{C}_{\bullet}(\Aad((\hb_1))((\hb_2)))|_{V} \rar \underline{\K}[2n]|_{V}$ induces an isomorphism on homologies. But the restriction $|_{V}$ factors through $j^*$. Hence, $\mu|_{V}$ coincides with the composite map
$$ \text{C}_{\bullet}(\Aad((\hb_1))((\hb_2)))|_{V} \rar \text{C}_{\bullet}(\Aa_{M^-}((\hb_1))((\hb_2)))|_{V} \stackrel{\chi_{FFS}}{\rar} \Omega^{2n-\bullet}_{M^-}((\hb_1))((\hb_2))|_V $$ in $\text{D}(\text{PrSh}_{\K}(V))$.
The first arrow in the above composition is a term by term isomorphism of complexes of presheaves on $V$. The second arrow is known to be a quasi-isomorphism. The reader may see ~\cite{EF} for the analogous assertion for the Hochschild chain complex of the sheaf of holomorphic differential operators on a compact, complex manifold. Also, a cyclic homology analog of this assertion is available in Section 5.2 of ~\cite{W}.\\

{\it Case 2: $V = N_{\alpha,U}$ for a sufficiently small open subset $U$ of $M_2^{\gamma}$.} This case is a little more involved. By Proposition~\ref{p2}, $ \text{C}_{\bullet}(\Aad((\hb_1))((\hb_2)))|_{V}$ is quasi-isomorphic to $\underline{\K}_V[2n] \oplus i_*\underline{\K}_U[2n-2]$ as complexes of presheaves on $V$. Since sheafification is exact, $ \text{C}^{sh}_{\bullet}(\Aad((\hb_1))((\hb_2)))|_{V}$ is quasi-isomorphic to $\underline{\K}_V[2n] \oplus i_*\underline{\K}_U[2n-2]$ with the natural adjunction from $ \text{C}_{\bullet}(\Aad((\hb_1))((\hb_2)))|_{V}$ to $ \text{C}^{sh}_{\bullet}(\Aad((\hb_1))((\hb_2)))|_{V}$ inducing the identity on homologies. For the rest of this proof, we will identify $ \text{C}_{\bullet}(\Aad((\hb_1))((\hb_2)))|_{V}$ as well as its sheafification with $\underline{\K}_V[2n] \oplus i_*\underline{\K}_U[2n-2]$ wherever it is convenient for us to do so. With this identification, one may view $\mu^{sh}|_V$ as a map in $\text{D}(\text{Sh}_{\K}(V))$ from $\underline{\K}_V[2n] \oplus i_*\underline{\K}_U[2n-2]$ to $\underline{\K}_V[2n]$. Note that
$$\text{Hom}_{\text{D}(\text{Sh}_{\K}(V))}(i_*\underline{\K}_U[2n-2],\underline{\K}_V[2n])=\text{Hom}_{\text{D}(\text{Sh}_{\K}(V))}(i_*\underline{\K}_U[2n-2],j_*\underline{\K}_{V \setminus U}[2n])
$$ $$=\text{Hom}_{\text{D}(\text{Sh}_{\K}(V \setminus U))}(j^*i_*\underline{\K}_U[2n-2],\underline{\K}_{V \setminus U}[2n])=0 \text{ .}$$
It follows that the composite map
$$i_*\underline{\K}_U[2n-2] \rar i_*\underline{\K}_U[2n-2] \oplus \underline{\K}_V[2n] \stackrel{\mu^{sh}}{\rar} \underline{\K}_V[2n]$$
is zero in $\text{D}(\text{Sh}_{\K}(V))$. On the other hand, the composite map
$$\underline{\K}_V[2n] \rar i_*\underline{\K}_U[2n-2] \oplus \underline{\K}_V[2n] \stackrel{\mu^{sh}}{\rar} \underline{\K}_V[2n] $$
is nonzero. This is because, applying $j^*$ to the above composite gives the composite map
$$ \underline{\K}_{V \setminus U}[2n] \rar \text{C}^{sh}_{\bullet}(\Aa_{M^-}((\hb_1))((\hb_2)))|_{V\setminus U} \stackrel{\chi_{FFS}}{\rar} \Omega^{2n-\bullet}_{M^-}((\hb_1))((\hb_2))|_{V \setminus U} \cong \underline{\K}_{V \setminus U}[2n] $$ which is the identity if the first quasi-isomorphism above is suitably chosen. It follows that the composite map
$$\underline{\K}_V[2n] \rar i_*\underline{\K}_U[2n-2] \oplus \underline{\K}_V[2n] \stackrel{\mu}{\rar} \underline{\K}_V[2n] $$
is an isomorphism in $\text{D}(\text{PrSh}_{\K}(V))$ and the composite map
$$i_*\underline{\K}_U[2n-2] \rar i_*\underline{\K}_U[2n-2] \oplus \underline{\K}_V[2n] \stackrel{\mu}{\rar} \underline{\K}_V[2n]$$ is zero in
$\text{D}(\text{PrSh}_{\K}(V))$. \\

The remaining task is to check that the composite map
$$\underline{\K}_V[2n] \rar i_*\underline{\K}_U[2n-2] \oplus \underline{\K}_V[2n] \stackrel{\lambda}{\rar} i_*\underline{\K}_U[2n-2]$$ is zero and that the composite map
$$i_*\underline{\K}_U[2n-2] \rar i_*\underline{\K}_U[2n-2] \oplus \underline{\K}_V[2n] \stackrel{\lambda}{\rar} i_*\underline{\K}_U[2n-2]$$ is an isomorphism in $\text{D}(\text{PrSh}_{\K}(V))$. We first note that by its very construction, $\lambda$ factors through $\text{C}^{sh}_{\bullet}(\Aad((\hb_1))((\hb_2)))|_{V}$. We may therefore, verify these facts in $\text{D}(\text{Sh}_{\K}(V))$. For the first of the above two facts, note that
$$\text{Hom}_{\text{D}(\text{Sh}_{\K}(V))}(\underline{\K}_V[2n],i_*\underline{\K}_U[2n-2])=\text{Hom}_{\text{D}(\text{Sh}_{\K}(U))}(\underline{\K}_U[2n],\underline{\K}_U[2n-2])=0 \text{.}$$
The second fact is immediate from Proposition~\ref{locdarb} and equation ~\eqref{localf2} in the example provided in Section 4.2.1.
\end{proof}

\subsection{The other trace.} As in the previous subsection, we assume that $M^{\gamma}=M_2^{\gamma}$ in order to avoid notational complexities. Our methods in this subsection work in the general situation with trivial modifications. The map
$$ {\mathbb H}(\mu): {\mathbb H}^{-\bullet}(M, \text{C}_{\bullet}(\Aad((\hb_1))((\hb_2)))) \rar \text{H}^{2n-\bullet}(M,\K) $$
induced by $\mu$ on hypercohomologies is $\Z_2$ equivariant. In particular, by Proposition~\ref{p7}, it induces a map
$$ \chi_0: \text{HH}_{\bullet}(\mathfrak{A}_{M/\Z_2}((\hb_1))((\hb_2))) \rar \text{H}^{2n-\bullet}(M,\K)^{\Z_2} \text{.}$$
The degree $0$-component of this map $\chi_0:\text{HH}_{0}(\mathfrak{A}_{M/\Z_2}((\hb_1))((\hb_2))) \rar \text{H}^{2n}(M,\K)^{\Z_2}$ is another trace on $\mathfrak{A}_{M/\Z_2}((\hb_1))((\hb_2))$. Unfortunately, since the construction of $\chi_0$ is not completely explicit, we are not in a position to prove an algebraic index theorem computing $\chi_0(1)$. It would indeed be interesting to have an explicit index formula computing $\chi_0(1)$. We however, discuss some basic facts about $\chi_0$ which enable us to prove a proposition completing the proof of Theorem~\ref{four}.\\

We begin with the observation that $\mu$ restricts to a well defined map
\[
\text{C}_{\bullet}(\Aad((\hb_1))[[\hb_2]])
\stackrel{\mu}{\longrightarrow}
\underline{\c}((\hb_1))[[\hb_2]][2n]\ \text{in}\
\text{D}(\text{PrSh}_{\c((\hb_1))[[\hb_2]]}(M)).
\]
By Proposition~\ref{p8}, $\chi_0$ ``restricts" to a map of graded
$\c((\hb_1))[[\hb_2]]$-modules
$$ \chi_0: \text{HH}_{\bullet}(\mathfrak{A}_{M/\Z_2}((\hb_1))[[\hb_2]]) \rar \text{H}^{2n-\bullet}(M,\c((\hb_1))[[\hb_2]])^{\Z_2} \text{.}$$
One has an algebra homomorphisms
$\mathfrak{A}_{M/\Z_2}((\hb_1))[[\hb_2]] \rar \A_{M/\Z_2}((\hb_1))$
induced by the map $\c((\hb_1))[[\hb_2]] \rar \c((\hb_1))$ taking
$\hb_2$ to $0$. The following proposition asserts that $\chi_0$
``quantizes" $\chi_{FFS}$.

\begin{prop} \la{trqt}
The following diagram commutes.
$$\begin{CD}
\text{HH}_{\bullet}(\mathfrak{A}_{M/\Z_2}((\hb_1))[[\hb_2]]) @>{\hb_2 \mapsto 0}>> \text{HH}_{\bullet}(\A_{M/\Z_2}((\hb_1)))\\
@VV{\chi_0}V          @V{\chi_{FFS}}VV\\
\text{H}^{2n-\bullet}(M,\c((\hb_1))[[\hb_2]])^{\Z_2} @>{\hb_2 \mapsto 0}>> \text{H}^{2n-\bullet}(M,\c((\hb_1)))^{\Z_2}\\
\end{CD}$$
\end{prop}

\begin{proof} The construction of $\mu$ can be imitated step by step with $\Aad$ replaced by $\Aa$ to obtain a map
$\nu: \text{C}_{\bullet}(\Aa((\hb_1))) \rar \underline{\c}((\hb_1))[2n]$ in $\text{D}(\text{PrSh}_{\c((\hb_1))}(M))$. One clearly has the following commutative diagram in $\text{D}(\text{PrSh}_{\c((\hb_1))}(M))$.
$$\begin{CD}
\text{C}_{\bullet}(\Aad((\hb_1))[[\hb_2]]) @>{\hb_2 \mapsto 0}>> \text{C}_{\bullet}(\Aa((\hb_1)))\\
@VV{\mu}V    @V{\nu}VV\\
 \underline{\c}((\hb_1))[[\hb_2]][2n] @>{\hb_2 \mapsto 0}>> \underline{\c}((\hb_1))[2n]\\
 \end{CD}$$
Further, one also has the following commutative diagram in $\text{D}(\text{Sh}_{\c((\hb_1))}(M))$.
$$\begin{CD}
\text{C}_{\bullet}(\Aa((\hb_1))) @>\text{id}>> \text{C}_{\bullet}(\Aa((\hb_1)))\\
@VV{\chi_{FFS}}V     @V{\nu^{sh}}VV\\
\underline{\c}((\hb_1))[2n] @>>> Rj_*\underline{\c}((\hb_1))[2n]\\
\end{CD}$$
The bottom arrow in the above diagram is the natural adjunction map. The desired proposition follows from the two commutative diagrams described above once we observe that the map $\text{H}^{\bullet}(M,\c) \rar \mathbb{H}^{\bullet}(M,Rj_*\underline{\c}) \cong \text{H}^{\bullet}(M,\c)$ induced via the adjunction $\underline{\c} \rar Rj_*\underline{\c} \cong \underline{\c}$ is the identity.
\end{proof}

We remark that with Proposition~\ref{trqt}, one can see without difficulty that $\chi_0$ and $\chi^{\gamma}$ are linearly independent over $\c((\hb_1))((\hb_2))$. Further, note that the argument in the proof of Proposition~\ref{qism2} can be easily modified to show that $\mu \oplus \lambda$ yields a quasiisomorphism between $\text{C}_{\bullet}(\Aad((\hb_1))[[\hb_2]])$ and $\underline{\c}((\hb_1))[[\hb_2]][2n] \oplus i_*\underline{\c}_{M^{\gamma}}((\hb_1))[[\hb_2]][2n-2]$ in
$\text{D}(\text{PrSh}_{\c((\hb_1))[[\hb_2]]}(M))$. Together with Proposition~\ref{p8}, this implies that
the trace map $\chi_0: \text{HH}_{\bullet}(\mathfrak{A}_{M/\Z_2}((\hb_1))[[\hb_2]]) \rar \text{H}^{2n-\bullet}(M,\c((\hb_1))[[\hb_2]])^{\Z_2}$ is a surjection.
By Proposition~\ref{trqt}, one immediately has the first part of the following proposition (the second part being obvious). The reader is reminded that $\theta$ and $\theta_0$ have the same meaning as in the proof of Theorem~\ref{four}.

\begin{prop} \la{vdbf}
There is a $2n$-cycle $\theta$ of $\mathfrak{A}_{M/\Z_2}((\hb_1))((\hb_2))$ such that $[\theta_0]=(\hb_2 \mapsto 0)[\theta] \neq 0$. Further, the following diagram commutes:
$$\begin{CD}
\text{HH}_{\bullet}(\mathfrak{A}_{M/\Z_2}((\hb_1))[[\hb_2]]) @>>> \text{HH}_{\bullet}(\mathfrak{A}_{M/\Z_2}((\hb_1))((\hb_2)))\\
@VV{\chi_0}V  @V{\chi_0}VV\\
\text{H}^{2n-\bullet}(M,\c((\hb_1))[[\hb_2]])^{\Z_2} @>>> \text{H}^{2n-\bullet}(M,\c((\hb_1))((\hb_2)))^{\Z_2}\\
\end{CD}$$

\end{prop}


\begin{thebibliography}{11}
%
\bibitem[AFLS]{AFLS} J. Alev, M. A. Farinati, T. Lambre and A. L. Solotar, \textit{Homologie des invariants d'une alg\`{e}bre de Weyl sous l'action d'une group fini}, J. Algebra \textbf{232} (2000), 564-577.
%
\bibitem[BEG]{BEG} Yu. Berest, P. Etingof and V. Ginzburg, \textit{Finite dimensional representations of rational Cherednik algebras}, IMRN \textbf{19} (2003), 1053-1088.
%
\bibitem[BEG1]{BEG1} Yu. Berest, P. Etingof and V. Ginzburg, \textit{Morita equivalence of Cherednik algebras}, J. Reine Angew. Math. \textbf{568} (2004), 81-98.
%
\bibitem[Co]{Co} A. Connes, \textit{Noncommutative differential geometry}, Publ. Math. IHES \textbf{62} (1985), 41-144.
%
\bibitem[DE]{DE} V. Dolgushev and P. Etingof, \textit{Hochschild cohomology of quantized symplectic orbifolds and the Chen-Ruan cohomology}, IMRN \textbf{27} (2005), 1657-1688.
%
\bibitem[E]{E} P. Etingof, {\it Cherednik and Hecke algebras of varieties
with a finite group
action}, Arxiv preprint math.QA/0406499.

%
\bibitem[EF]{EF} M. Engeli and G. Felder, \textit{A Riemann-Roch-Hirzebruch formula for traces of differential operators}, Ann. Sci. \'{E}c. Norm. Sup\'{e}r. (4) \textbf{41}(2008), No. 4, 621-653.
%
\bibitem[EG]{EG} P. Etingof and V. Ginzburg, \textit{Sympectic reflection algebras, Calogero-Moser space, and deformed Harishchandra homomorphism}, Invent. Math. \textbf{147} (2002), 243-348.
%
\bibitem[F]{F} B. V. Fedosov, \textit{On $G$-trace and $G$-index in deformation quantization}, Letters in Math. Phys. \textbf{52} (2000), 29-49.
%
\bibitem[F1]{F1} B. V. Fedosov, \textit{Deformation quantization and index theory}, Akademie Verlag, Berlin, 1996.
%
\bibitem[FFS]{FFS} B. Feigin, G. Felder and B. Shoikhet, \textit{Hochschild cohomology of the Weyl algebra and traces in deformation quantization}, Duke Math. J. \textbf{127} (2005), No. 3, 487-517.
%
\bibitem[G]{G} V. Ginzburg, \textit{Calabi-Yau algebras}, Arxiv preprint arxiv:math/0612139v3.
%
\bibitem[HT]{HT} G. Halbout and X. Tang, \textit{Dunkl operator and quantization of $\Z_2$-singularity}, Arxiv preprint arxiv:0908.4301.
%
\bibitem[K]{K} O. Kravchenko, \textit{Deformation quantization of symplectic fibrations}, Compositio Math. \textbf{123} (2000), No. 2, 131-165.
%
\bibitem[Menc]{Menc} L. Menichi, \textit{Van Den Bergh type isomorphisms in string topology}, Arxiv preprint arxiv:0907.2105v3.
%
\bibitem[MM]{MM} X. Ma, G. Marinescu, \textit{Toeplitz operators on symplectic
manifolds}, J. Geom. Anal. 18, 565-611, (2008).
%
\bibitem[NPPT]{NPPT} N. Neumaier, M. Pflaum, H. Posthuma and X. Tang, \textit{Homology of formal deformations of proper \'{e}tale groupoids}, J. Reine Angew. Math. \textbf{593}, 117-168 (2006)
%
\bibitem[P]{P} M. Pflaum, \textit{\it On the deformation quantization of
symplectic orbispaces},
  Diff. Geom. Applications \textbf{19}, 343--368 (2003).
%
\bibitem[PPT]{PPT} M. Pflaum, H. Posthuma and X. Tang, \textit{An algebraic index theorem for orbifolds}, Adv. Math. \textbf{210} (2007), No. 1, 83-121.
%
\bibitem[PPT2]{PPT2} M. Pflaum, H. Posthuma and X. Tang, \textit{Cyclic cocycles in deformation quantization and higher index theorems}, Adv. Math. \textbf{223} (2010), No. 6, 1958-2021.
%
\bibitem[PPTT]{PPTT}  M. Pflaum, H. Posthuma, X. Tang and H. Tseng, \textit{Orbifold cup products and ring structures on Hochschild cohomologies}, Arxiv preprint arxiv:0706.0027, to apear in Comm. Contemp. Math.
%
\bibitem[R]{R} A. Ramadoss, \textit{Integration of cocycles and Lefschetz number formulae for differential operators}, Arxiv preprint arxiv:0904.1891.
%
\bibitem[S]{S} Shoikhet, B., \textit{Integration of the Lifting formulas and the Cyclic homology of the algebra of differential operators},
    {Geom. Funct. Anal.} \textbf{11} (2001), 1096-1124.
%
\bibitem[T]{T}
  X.~Tang, \textit{Deformation quantization of pseudo Poisson
  groupoids}, Geom.~Func.~Analysis \textbf{16} (2006) No.~3, 731--766.
%
\bibitem[Te]{Te} N. Teleman,\textit{Microlocalisation de l'homologie de Hochschild}, C. R. Acad. Sci. Paris \textbf{326} (1998), 1262-1264.
%
\bibitem[Vdb1]{Vdb1} M. Van Den Bergh, \textit{A relation between Hochschild homology and cohomology for Gorenstein rings}, Proc. AMS \textbf{126} (1998), No. 5, 1345-1348.
%
\bibitem[Vdb2]{Vdb2} M. Van Den Bergh, \textit{Erratum to "A relation between Hochschild homology and cohomology for Gorenstein rings"}, Proc. AMS \textbf{130} (2002), No. 9, 2809-2810.
%
\bibitem[W]{W} T. Willwacher, \textit{Cyclic Cohomology of the Weyl algebra}, Arxiv preprint Arxiv:0804.2812
\end{thebibliography}
\end{document}